\documentclass[12pt]{amsart}

\usepackage{lipsum} 
\usepackage{amsfonts}
\usepackage{amsmath}
\usepackage{amssymb}
\usepackage{amsthm} 
\usepackage{mathrsfs}
\usepackage{xcolor}
\usepackage{appendix}

\usepackage[colorlinks=true,urlcolor=blue,linkcolor=blue,citecolor=black]{hyperref}

\usepackage{verbatim}

\usepackage[retainorgcmds]{IEEEtrantools}

\renewcommand{\AA}{\mathcal{A}}

\newcommand{\vei}{\left\vert}
\newcommand{\ved}{\right\vert}
\newcommand{\Vei}{\left\Vert}
\newcommand{\Ved}{\right\Vert}

\theoremstyle{definition}
\newtheorem{thm}{Theorem}[section] 
\newtheorem{cor}[thm]{Corollary}
\newtheorem{lem}[thm]{Lemma}
\newtheorem{prop}[thm]{Proposition}
\theoremstyle{definition}
\newtheorem{defn}[thm]{Definition}
\newtheorem{rem}[thm]{Remark}
\newtheorem{nota}[thm]{Notation}

\newtheorem{conj}[thm]{Conjecture}

\numberwithin{equation}{section}  

\title[On Operator-Valued Infinitesimal Boolean and Monotone Independence]{On Operator-Valued Infinitesimal\\ Boolean and Monotone Independence}

\author{Daniel Perales}
\email{dperales@uwaterloo.ca}
\address{Department of Mathematics\\
University of Waterloo}
\author{Pei-Lun Tseng}
\email{16plt@queensu.ca}
\address{Department of Mathematics and Statistics\\
Queen's University}

\begin{document}
\begin{abstract}

We introduce the notion of operator-valued infinitesimal (OVI) independence for the Boolean and monotone cases. Then show that OVI Boolean (resp. monotone) independence is equivalent to the operator-valued Boolean (resp. monotone) independence over an algebra of $2\times 2$ upper triangular matrices. Moreover, we derive formulas to obtain the OVI Boolean (resp. monotone) additive convolution by reducing it to the operator-valued case. 

We also define OVI Boolean and monotone cumulants and study its basic properties. Moreover, for each notion of OVI independence, we construct the corresponding OVI Central Limit Theorem. The relations among free, Boolean and monotone cumulants are extended to this setting. Besides, in the Boolean case we deduce that the vanishing of mixed cumulants is still equivalent to independence, and use this to connect scalar-valued with matrix-valued infinitesimal Boolean independence.  Finally we study two random matrix models that are asymptotically Boolean independent but turn out to not be infinitesimally Boolean independent.
\end{abstract}

\maketitle
\tableofcontents

\section{Introduction}

In non-commutative probability theory, the classical notion of independence is replaced by other analogue notions that are better suited for the non-commutative framework. If we ask these notions to satisfy certain canonical properties, we end up with only five types of independence (see \cite{Mur3}), which are tensor (or classical), free \cite{Voi}, Boolean \cite{SW}, monotone \cite{Mur1}, and anti-monotone.  Non-commutative probability spaces consist of a unital algebra $\mathcal{A}$ and a functional $\varphi:\mathcal{A}\to\mathbb{C}$ that maps the unit $1_\mathcal{A}$ to $1$. The functional $\varphi$ replaces the classical notion of expectation. Freeness was introduced by Voiculescu in the 1980s and since then many extensions and variations (such as Boolean and monotone) have been studied. Out of the many possible extensions, we want to focus on two, infinitesimal freenees and operator valued freeness.

For infinitesimal freeness, besides our probability space $(\mathcal{A},\varphi)$ we consider another linear functional $\varphi':\mathcal{A}\to\mathbb{C}$ satisfying $\varphi'(1_\mathcal{A})=0$. Such triple $(\mathcal{A},\varphi,\varphi')$ is called an infinitesimal non-commutative probability space. $\varphi'$ can be interpreted as the derivative of a continuous family of linear functionals $\{\varphi_t\}_{t\in T}$ for some set $T$ with $0$ as an accumulation point and $\varphi_0=\varphi$.  This notion was developed in \cite{BS}, (see also \cite{BGN}), and was used by Shlyakhtenko \cite{Shl} to understand the finite rank perturbations in some random matrix models. Infinitesimal freeness can be studied applying free probability tools in $2\times 2$ upper triangular matrices, see for instance \cite{BGN} and \cite {BS}, or equivalently using the two-dimensional Grassman algebra, see \cite{FN}. The infinitesimal analogue of Boolean and monotone independences was studied in \cite{Has}.

On the other hand, operator-valued freeness (or freeness with amalgamation) parallels the classical concept of conditional independence. Here we first consider a unital subalgebra $\mathcal{B}$ of $\mathcal{A}$, and then replace $\varphi$ with a $\mathcal{B}$-bimodule linear map $E:\mathcal{A}\to \mathcal{B}$. This structure was first introduced by D. Voiculescu in \cite{voi3}, and this theory was successfully applied by Belinschi, Mai, and Speicher \cite{BMS} to describe the limit behavior of rational functions of some random matrix models via linearization trick. Boolean and monotone independences in the operator-valued framework have been also studied, see for instance \cite{popa2008combinatorial}, \cite{popa2009new}, \cite{HS14} and references therein. 

A natural question is what happens when we study infinitesimal freeness in an operator-valued framework. This was addressed by Curran and Speicher \cite{cur1}, where a random matrix model of asymptotic infinitesimal freeness with amalgamation is also provided. In \cite{tseng2019operator}, the second author established a connection between the notion of operator-valued infinitesimal freeness and freeness over an algebra of $2\times 2$ upper triangular matrices. Using this characterization, the author gave a explicit formula to find the operator-valued infinitesimal free additive convolution using a infinitesimal version of the Cauchy transform.

The main purpose of this paper is to investigate the notions of infinitesimal Boolean and monotone independence in the operator-valued framework. This is mainly done by following the development for the free case already done by the second author in \cite{tseng2019operator}. Since we are in a operator-valued infinitesimal (OVI) framework we will consider a probability space of the form $(\mathcal{A},\mathcal{B},E,E')$ where $\mathcal{B}$ is a unital subalgebra of $\mathcal{A}$ and $E,E':\mathcal{A}\to\mathcal{B}$ are linear and satisfy some other basic relations that are made precise in Definition \ref{def.ovips}. In this setting it is easy to extend the notion of Boolean and monotone independence. Then, we consider the upper triangular probability space $(\widetilde{\mathcal{A}},\widetilde{\mathcal{B}},\widetilde{E})$ associated to $(\mathcal{A},\mathcal{B},E,E')$ and show that operator-valued Boolean (respectively monotone) independence in the former is equivalent to operator-valued infinitesimal Boolean (respectively monotone) independence in the latter (see Proposition \ref{p1} and Proposition \ref{p2}). Moreover these results allow us to compute the OVI Boolean  and monotone additive convolution, by reducing the problem to the operator-valued framework. Specifically, given two infinitesimal Boolean independent selfadjoint random variables $x,y$ in a $C^*$-operator-valued infinitesimal probability space, the infinitesimal Cauchy transform $g_{x+y}$ can be written in terms of the inifinitesimal Cauchy transform ($g_x$, $g_y$) and the $F$-transform ($F_x$, $F_y$) of each $x$ and $y$ (we refer the reader to Section \ref{sec:preliminary} for the precise definitions of all these notions).

\begin{thm}
\label{Thm.Boolean.add.conv}
Suppose that $(\mathcal{A},\mathcal{B},E,E')$ is a $C^*$-operator-valued infinitesimal probability space, and let $x,y\in\mathcal{A}$ be self-adjoint and infinitesimally Boolean independent over $\mathcal{B}$. Then for all $b\in H^+(\mathcal{B})$,   
\begin{equation*}
g_{x+y}(b) = ( F_x(b)+F_y(b)-b)^{-1} \Big( F_x(b)g_x(b)F_x(b) + F_y(b)g_y(b)F_y(b) \Big) (F_x(b)+F_y(b)-b)^{-1}. 
\end{equation*}
\end{thm}

For the monotone convolution we have a similar result. In this case we write $g_{x+y}$ in terms of $g_x,g_y,F_x$ and $G'_x$. The latter being the derivative of the Cauchy transform of $x$.

\begin{thm}
\label{Thm.monotone.add.conv}
Suppose that $(\mathcal{A},\mathcal{B},E,E')$ is a $C^*$-operator-valued infinitesimal probability space, and let $x,y\in\mathcal{A}$ be self-adjoint and infinitesimally monotone independent over $\mathcal{B}$. Then for all $b\in H^+(\mathcal{B})$, 
$$
g_{x+y}(b) = G_x'(F_y(b))(-F_y(b)g_y(b)F_y(b))+g_x(F_y(b)).
$$
\end{thm}

Recall that given $(\mathcal{A},\varphi,\varphi')$, a infinitesimal non-commutative probability space, $\varphi$ and $\varphi'$ can be seen as limits of a family of linear functionals $\{\varphi_t\}_{t\in T}$ and its derivatives. This prompted the use of differentiable paths $\mu(t)$ to study the infinitesimal Cauchy transform and infinitesimal free additive convolution (see details in \cite{BS}). This idea was generalized to study the OVI free additive convolution in \cite{tseng2019operator}. We devote Section \ref{sec:differentiable.paths} to study the OVI Boolean and monotone additive convolution using the differential paths approach, the main result using this method can be found in Theorem \ref{thm.diff.paths.main}. 

Regarding the combinatorial approach to freeness, the fundamental tool to study the additive convolution are the free cumulants, introduced by Speicher in \cite{Spe94}. The corresponding notions of Boolean and monotone cumulants were studied long ago (see \cite{SW} and \cite{HS1}). Explicit relations among Boolean, free and monotone cumulants were studied in detail in \cite{AHLV}. These relations have proven to be important in the study of non-commutative probability theory, (see for instance \cite{BN2}). The appropriate notions of cumulants for each of free, Boolean and monotone cases have been already defined in the operator-valued probability setting (see \cite{Spe98}, \cite{popa2009new} \cite{HS14}), as well as in the infinitesimal probability case (see \cite{FN}, \cite{Has}). The relations among them in the infinitesimal case were studied in \cite{celestino2019relations}. And the relation between Boolean and free cumulants in the operator valued setting can be found in \cite{ABFN}.

In this paper, we define the notion of Boolean and monotone cumulants in the OVI framework and study its basic properties. For the Boolean case, the fundamental property of vanishing of mixed Boolean cumulants still holds (see Theorem \ref{thm3}) and this allows to carry over scalar-valued infinitesimal Boolean independence into operator-valued infinitesimal Boolean independence. 

\begin{thm}
\label{thm.relation.inf.boo.matrices}
Suppose that $(\mathcal{M},\varphi,\varphi')$ is an infinitesimal probability space. If the sets 

$\{a^{(1)}_{i,j}\colon1\le i,j\le N\},\dots, \{a^{(k)}_{i,j}\colon1\le i,j\le N\}$ are infinitesimally Boolean in $\mathcal{M}$, then the elements $[a^{(1)}_{i,j}]_{i,j=1}^N,\dots, [a^{(k)}_{i,j}]_{i,j=1}^N$ are infinitesimally Boolean in $M_N(\mathcal{M})$ over $M_N(\mathbb{C})$.
\end{thm}

Using the vanishing of infinitesimal Boolean cumulants (restricted to the scalar case) we study two random matrix models that yield asymptotic Boolean independence, an thus were natural candidates to yield asymptotic infinitesimal Boolean independence. However, it turns out this is not the case and neither of the models achieve infinitesimal independence. Computations for each model can be found in the Appendix.

Once we have the concepts of free, Boolean and monotone OVI cumulants, we can study the relations among them. In principle, there is no reason why the same formulas as in the scalar-valued case should hold, even for the operator-valued framework. However, after proving some technical results (See Lemma \ref{lemma10.Tseng} and Lemma \ref{Lemma.useful}) we can observe that the relations just rely on the underlying combinatorial structure, rather than in the intricate way to define operator-valued cumulants. Thus everything works smoothly and we can parallel the previous approaches to get operator valued version. Moreover, using the associated upper triangular space we can derive the relations in the OVI case too. 

\begin{thm}
\label{Thm.relation.cumulants}
Suppose that $(\mathcal{A},\mathcal{B},E,E')$ is an OVI probability space, and let $\{ r^{\mathcal{B}}_n\}_{n \ge 1}$, $\{\beta^{\mathcal{B}}_n\}_{n \ge 1}$, and $\{h^{\mathcal{B}}_n\}_{n \ge 1}$ be the families of free, Boolean and monotone OV cumulants, respectively. Also let $\{\partial r^{\mathcal{B}}_n\}_{n \ge 1}$, $\{\partial \beta^{\mathcal{B}}_n\}_{n \ge 1}$, and $\{\partial h^{\mathcal{B}}_n\}_{n \ge 1}$ be the families of free, Boolean and monotone OVI cumulants, respectively. Then, the following relations among them hold for every $n$:
\begin{IEEEeqnarray}{rclrcl}
     	\partial\beta^{\mathcal{B}}_n
		&=& \sum_{\pi\in \mathcal{NC}_{irr}(n)} \partial r^{\mathcal{B}}_\pi, \label{OVIBooFree} \hspace{3cm} &
		\partial r^{\mathcal{B}}_n 
		&=& \sum_{\pi\in \mathcal{NC}_{irr}(n)} (-1)^{|\pi|-1} \partial \beta^{\mathcal{B}}_\pi, \label{OVIFreeBoo}\\
	\partial \beta^{\mathcal{B}}_n  
		&=& \sum_{\pi\in \mathcal{NC}_{irr}(n)} \frac{1}{\tau(\pi)!}\partial h^{\mathcal{B}}_\pi, \label{OVIBooMon} &
	\partial r^{\mathcal{B}}_n  
		&=& \sum_{\pi\in \mathcal{NC}_{irr}(n)} \frac{(-1)^{|\pi|-1}}{\tau(\pi)!}\partial h^{\mathcal{B}}_\pi, \label{OVIFreeMon}\\
		\partial	h^{\mathcal{B}}_n  
		&=& \sum_{\pi\in \mathcal{NC}_{irr}(n)} \omega(\pi) \partial \beta^{\mathcal{B}}_\pi, \label{OVIMonBoo} &
	\partial h^{\mathcal{B}}_n  
		&=& \sum_{\pi\in \mathcal{NC}_{irr}(n)} (-1)^{|\pi|-1}\omega(\pi) \partial r^{\mathcal{B}}_\pi, \label{OVIMonFree}  
    \end{IEEEeqnarray}
where $\mathcal{NC}_{irr}(n)$ is the set of non-crossing irreducible paritions, and $\tau(\pi)!$, $\omega(\pi)$ are just some real values assigned to each non-crossing partition that are explicitly defined in \eqref{eq.tree.factorial} and \eqref{eq.muruaomega}, respectively.
\end{thm}

Besides this introductory section, this paper has five other sections. Section \ref{sec:preliminary} contains the basic notions of Operator-Valued and OVI probabity spaces, freeness and free cumulants. In Section \ref{sec:Boolean.Indep} we deal with the Boolean case, we define of OVI independence and cumulants, study its basic properties and prove Theorem \ref{Thm.Boolean.add.conv} and Theorem  \ref{thm.relation.inf.boo.matrices}. The same is done for the monotone case in Section \ref{sec:Monotone.indep}, were we prove Theorem \ref{Thm.monotone.add.conv}.

In Section \ref{sec:InfCLT}, for OVI free, Boolean, and monotone independences, we construct the corresponding OVI Central Limit Theorem (Theorem \ref{OVICLT}). Furthermore, we provide another approach to study the scalar-valued Infinitesimal Central Limit Theorem, see Subsection \ref{ssec:SVICLT}.

In Section \ref{sec:Relation.cumulants} we study the relation among OVI cumulants and prove Theorem \ref{Thm.relation.cumulants}. Finally, in Section \ref{sec:differentiable.paths} we use the differentiable paths approach to study boolean and monotone convolution. Appendix \ref{sec:antitrace.model} and Appendix \ref{sec:partal.trace.model} cover the study of two random matrix models that although being asymptotic Boolean independent, they are not asymptotically infintesimally Boolean independent.\\

Acknowledgements:  The authors would like to thank professors Belinschi, Mingo and Nica for their valuable comments. The first author expresses his gratitude to CONACyT (Mexico) for its support via the scholarship 714236.\\

\section{Preliminaries}
\label{sec:preliminary}


\subsection{Notions of Operator-Valued Non-Commutative Independences}

$\ $

\noindent Let $\mathcal{A}$ be a unital algebra, and $\mathcal{B}$ be a unital subalgebra of $\mathcal{A}$. A linear map $E:\mathcal{A}\rightarrow \mathcal{B}$ is a \emph{conditional expectation} if 
\begin{equation*}
E(b)=b \text{ and }  E(bab')=bE(a)b' 
\end{equation*}
for all  $b,b'\in\mathcal{B}$ and $a\in \mathcal{A}.$ Then the triple $(\mathcal{A},\mathcal{B},E)$ is called an \emph{operator-valued probability space}.
For a given random variable $x\in \mathcal{A}$,  the \emph{operator-valued distribution} of $x$ is the set of all operator-valued moments 
$$
E(xb_1xb_2\cdots b_{n-1}xb_nx)\in\mathcal{B}
$$
where $n\in\mathbb{N}$ and $b_1,\cdots,b_n\in \mathcal{B}$, and we denote it by $\mu_x$ (see \cite{voi3}).  
\begin{defn}
Let $(\mathcal{A},\mathcal{B},E)$ be an operator-valued probability space.
\begin{itemize}
  \item [(1)] 
  Subalgebras $(\mathcal{A}_i)_{i\in I}$ of $\mathcal{A}$ that contain $\mathcal{B}$ is \emph{freely independent} over $\mathcal{B}$ if for all $n\in \mathbb{N}$, $a_1,\dots, a_n\in \mathcal{A}$ are such that $a_j\in\mathcal{A}_{i_j}$ where $i_1,\dots,i_n\in I$, $i_1\neq \cdots \neq i_n$, and $E(a_j)=0$ for all $j$, then      
  $$
  E(a_1\cdots a_n) = 0.
  $$
  \item [(2)]
  Subalgebras $(\mathcal{A}_i)_{i\in I}$ of $\mathcal{A}$ that contain $\mathcal{B}$ is \emph{Boolean independent} over $\mathcal{B}$ if for all $n\in \mathbb{N}$, $a_1,\dots, a_n\in \mathcal{A}$ are such that $a_j\in\mathcal{A}_{i_j}$ where $i_1,\dots,i_n\in I$, $i_1\neq \cdots \neq i_n$, then
  $$
  E(a_1\cdots a_n)=E(a_1)\cdots E(a_n).
  $$
  \item [(3)]
  Assume $I$ is equipped with a linear order $<$. Subalgebras $(\mathcal{A}_i)_{i\in I}$ of $\mathcal{A}$ that contain $\mathcal{B}$ is \emph{monotone independent} over $\mathcal{B}$ if  
  $$
  E(a_1\cdots a_{j-1}a_ja_{j+1}\cdots a_n) = E(a_1\cdots a_{j-1}E(a_j)a_{j+1}\cdots a_n) 
  $$
  whenever $a_j\in\mathcal{A}_{i_j}$, $i_1,\dot,i_n\in I$ and $i_{j-1}<i_j>i_{j+1}$ where one of the inequalities is eliminated if $j=1$ or $j=n$.
\end{itemize}
\end{defn}
Given an operator-valued probability space $(\mathcal{A},\mathcal{B},E)$, $(x_i)_{i \in I}$ in $\mathcal{A}$ are said to be freely independent (resp, Boolean, monotone) over $\mathcal{B}$ if the unital algebras $alg(1,x_i)$ generated by elements $x_i$ are freely independent (resp, Boolean, monotone). 

If $x,y\in\mathcal{A}$ are freely (resp, Boolean, monotone) independent, then the distribution of $x+y$ only depends on $\mu_x$ and $\mu_y$, we shall call it \emph{free (resp Boolean, monotone) additive convolution} of $\mu_x$ and $\mu_y$, and denoted it by $\mu_x\boxplus \mu_y$ (resp, $\mu_x\uplus\mu_y,\  \mu_x\triangleright\mu_y$). 

From now on, given a operator-valued probability space $(\mathcal{A},\mathcal{B},E)$, we will equip it with an analytic structure. Assume $\mathcal{A}$ is a unital $C^*$-algebra, $\mathcal{B}$ is a unital $C^*$-subalgebra of $\mathcal{A}$, and $E:\mathcal{A}\rightarrow \mathcal{B}$ is a linear, completely positive, and conditional expectation map. Such triple $(\mathcal{A},\mathcal{B},E)$ is called a \emph{$C^*$-operator-valued probability space}. For a positive and invertible random variable $x\in\mathcal{A}$ we denote it by $x>0$, and then the \emph{operator upper-half plane} is defined by 
$$
H^+(\mathcal{B}) =\{ b\in \mathcal{B} \mid \operatorname{Im}(b)=\frac{1}{2i}(b-b^*)>0 \}.
$$
We note that if $x\in \mathcal{A}$ is self-adjoint, then $b-x$ is invertible for all $b\in H^+(\mathcal{B}).$ Now, for a fixed $x=x^*\in \mathcal{A}$, the \emph{Cauchy transform} of $x$ is defined by 
$$
G_x(b)= E [(b-x)^{-1}]
$$
for all $b\in \mathcal{B}$ so that $b-x$ is invertible. Note that $G_{x}$ is a holomorphic map that sends $H^{+}(\mathcal{B})$ into $H^-(\mathcal{B}):=-H^+(\mathcal{B})$. 
For each $n\in \mathbb{N}$, if we consider its fully matricial extension $G_x^{(n)}$ which is defined by
$$
G_{x}^{(n)}(b)=E\otimes 1_{n}[(b-x\otimes 1)^{-1}]
$$
for $b\in M_n(\mathcal{B})$ that $b-x\otimes 1$ is invertible, then it is known that the sequence $\{G_x^{(n)}\}_{n=1}^\infty$ encodes the distribution $\mu_x$. Note that $G_x^{(n)}$ maps the upper half plane $H^+(M_n(\mathcal{B}))$ into the lower half plane $H^-(M_n(\mathcal{B}))$, and $G_x^{(n)}$ on $H^+(M_n(\mathcal{B}))$ basically has the same behavior of $G_x=G_x^{(1)}$ on $H^+(\mathcal{B})$. We shall restrict our analysis on to $G_x$.  

Now, for a given selfadjoint random variable $x\in \mathcal{A}$, we define the following transforms:
\begin{eqnarray*}
&\text { the \emph{$R$-transofrm} of } x:& R_x(b)= G^{\langle -1 \rangle}_x(b)-b^{-1}  ; \\ 
&\text { the \emph{$F$-transform} of } x:& F_x(b)= G_x(b)^{-1}  ; \\
&\text { the \emph{$B$-transform} of } x:& B_x(b)= b-F_x(b)   
 \end{eqnarray*}
 where $G^{\langle -1 \rangle}$ is the compositional inverse of $G_x$ in a neighborhood of the origin. Then we have the following results. \\
If $x$ and $y$ are freely independent over $\mathcal{B}$, then for all $b$ in a neighborhood of origin, we have 
$$
R_{x+y}(b)=R_x(b)+R_y(b)\qquad \text{ (see \cite{voi3})}.
$$
If $x$ and $y$ are Boolean independent over $\mathcal{B}$, then for all $b\in H^+(\mathcal{B})$, we have 
\begin{equation}\label{Beqn}
B_{x+y}(b)=B_x(b)+B_y(b)\qquad \text{ (see \cite{popa2009new})}. 
\end{equation}
If $x$ and $y$ are monotone independent over $\mathcal{B}$, then for all $b\in H^+(\mathcal{B})$, we have 
\begin{equation} \label{Meqn}
F_{x+y}(b)=F_x(F_y(b))\qquad \text{ (see \cite{popa2008combinatorial})}. 
\end{equation}
Note that these formulas also hold in the algebraic setting, and we will apply \eqref{Beqn} and \eqref{Meqn} to deduce the infinitesimal convolutions in the next two sections.   
\subsection{Notions of Operator-Valued Cumulants}

$\ $

\noindent We will denote by $NC(n)$ the set of non-crossing partitions of $[n]:=\{1,2,\dots,n\}$, namely, partitions $\pi=\{V_1,\dots,V_r\}$ such that $V_i$ and $V_j$ do not cross for any $1\leq i,j\leq r$ (see \cite{NS}). The \emph{blocks} of $\pi$ are $V_1,\dots,V_r$ and the \emph{size} (in this case $r$) of the partition is denoted by $|\pi|$. We will used the partial order in $NC(n)$ given by the reverse-refinement, namely $\sigma\leq \pi$ if every block of $\sigma$ is contained in some block of $\pi$. We denote by $I(n)$ the subset of $NC(n)$ that consists of all the interval partitions of $[n]$. For a detailed introduction to non-crossing partitions we refer the reader to \cite{NS}.

\begin{defn}
\label{partitions.multiplicative}
Let $\{f_n\}_{n\geq 1}$ be a sequence of multilinear maps: $f_n:\mathcal{A}^n\to \mathcal{B}$. Then for every $\pi\in NC(n)$ we define the maps $f_\pi: \mathcal{A}^n\to \mathcal{B}$ recursively as follows:
\begin{enumerate}
    \item For $\pi=1_n$, we set $f_\pi=f_n$.
    \item For $\pi\in NC(n)$ we pick $V=\{l+1,\dots l+k\}\in \pi$ a interval block of $\pi$ (it always exist one) and set
    $$f_\pi(x_1,\dots,x_n)= f_{\pi'}(x_1,\dots,x_l f_k(x_{l+1},\dots,x_{l+k}),x_{l+k+1},\dots,x_n),$$
    where $\pi'=\pi\backslash V\in NC(n-k)$ is the partition obtained by deleting from $\pi$ the block $V$.
\end{enumerate}
It can be seen that $f_\pi$ does not depend on how we pick the interval block $V$. A family $\{f_\pi\}_{\pi\in NC}$ created in this way is called \emph{multiplicative}.
\end{defn}

\begin{nota}
We denote by $\{E_\pi\}$ the multiplicative family obtained using above construction with the sequence $\{E_n\}_{n\geq 1}$ where $E_n(x_1,\dots,x_n):=E(x_1\cdots x_n)$. 
\end{nota}

\begin{rem}
\label{rem.spei.mob}
Recall from \cite{Spe98} that if $\{f_\pi\}_{\pi\in NC}$ is a multiplicative family of $\mathcal{B}$-bimodule maps, and we define a new family 
$$g_\pi= \sum_{\substack{\sigma\in NC(n)\\ \sigma\leq \pi}} f_\sigma, \qquad \forall \pi \in NC(n).$$
Then the family $\{g_\pi\}_{\pi\in NC(n)}$ is also multiplicative. Moreover, Möbius inversion formula tell us that previous formula is equivalent to
$$f_\pi=\sum_{\substack{\sigma\in NC(n)\\ \sigma\leq\pi}} Mob(\sigma,\pi) g_\sigma,$$
where $Mob$ is the Möbius function on the lattice $NC(n)$. It also holds that multiplicativity of $\{g_\pi\}_{\pi\in NC(n)}$ implies multiplicativity of $\{f_\pi\}_{\pi\in NC(n)}$. This formula is important since a sum over non-crossing partitions is what links moments and free cumulants as we will see below. In Lemma \ref{Lemma.useful} we obtain similar formulas to deal with Boolean and monotone cumulants.
\end{rem}

\begin{defn}
The \emph{operator valued free cumulants} $\{r^{\mathcal{B}}_n:\mathcal{A}^n\to\mathcal{B}\}_{n\geq 1}$, the \emph{operator valued Boolean cumulants} $\{\beta^{\mathcal{B}}_n:\mathcal{A}^n\to\mathcal{B}\}_{n\geq 1}$ and the  \emph{operator valued monotone cumulants} $\{h^{\mathcal{B}}_n:\mathcal{A}^n\to\mathcal{B}\}_{n\geq 1}$ are recursively defined via the following moment-cumulant formulas:
\begin{equation}
\label{mom-free}
    E_n(a_1,\dots, a_n)= \sum_{\pi \in NC(n)} r_\pi^{\mathcal{B}}(a_1,\dots,a_n), \qquad \forall n\geq 1, a_1,\dots, a_n\in \mathcal{A},
\end{equation}

\begin{equation}
\label{mom-boo}
E_n(a_1,\dots, a_n)= \sum_{\beta \in I(n)} \beta_\pi^{\mathcal{B}}(a_1,\dots,a_n), \qquad \forall n\geq 1, a_1,\dots, a_n\in \mathcal{A},
\end{equation}

\begin{equation}
\label{mom-mon}
E_n(a_1,\dots, a_n)= \sum_{\beta \in NC(n)} \frac{1}{\tau(\pi)!} h_\pi^{\mathcal{B}}(a_1,\dots,a_n), \qquad\forall n\geq 1, a_1,\dots, a_n\in \mathcal{A},
\end{equation}
where $\tau(\pi)!$ is the tree factorial of the tree associated to the partition $\pi$, and can be defined directly in terms of $\pi$ as follows:
\begin{equation}
\label{eq.tree.factorial}
\tau(\pi)!:= \prod_{V\in\pi} \vei\Big\{ W\in \pi: W\subset [\min(V),\max(V)]\Big\}\ved,
\end{equation}
That is, for every block of $V\in\pi$ we count all the blocks nested inside $V$ (including it) and we multiply all these numbers.
\end{defn}

\begin{rem}
Equivalently, we can define the free and Boolean cumulants via the formulas
\begin{equation}
\label{free-mom}
r^{\mathcal{B}}_\pi(x_1,\dots,x_n)=\sum_{\pi\in NC(n)} Mob(\pi,1_n) E_\pi(x_1,\dots,x_n), \qquad \forall n\geq 1,  x_1,\dots, x_n\in \mathcal{A},
\end{equation}

\begin{equation}
\label{boo-mom}
\beta_\pi^{\mathcal{B}}(x_1,\dots,x_n)=\sum_{\pi\in NC(n)} (-1)^{|\pi|-1} E_\pi(x_1,\dots,x_n),  \qquad \forall n\geq 1,  x_1,\dots, x_n\in \mathcal{A},
\end{equation}
where $Mob$ is the Möbius function in the lattice $NC(n)$, and similarly, the coefficient $(-1)^{|\pi|-1}$ is obtained when performing Möbius inversion on the lattice $I(n)$. We should mention that there is no known direct formula to define monotone cumulants, however the fact they are uniquely defined is clear from the recursion that defines them. The original definition is different from the given above, but the formulas are equivalent, (see Theorem 3.4 in \cite{HS14}).
\end{rem}

Similar to the scalar-valued case, the operator-valued cumulants are very useful when studying the Central Limit Theorem (CLT). A detailed exposition on this result can be found in \cite{BMV} and \cite{Jek18}. In terms of cumulantss, this result can be concisely stated as follows.

\begin{thm}[Operator-Valued Central Limit Theorem]
\label{OVCLT}
Let $(\mathcal{A},\mathcal{B},E)$ be a $C^*$-operator-valued probability space and let $a,a_1,a_2,\dots$ be a sequence of self adjoint element identically distributed elements of $\AA$. By this we mean that $E(a_ib_1a_ib_2\cdots a_ib_{k-1}a_i)=E(ab_1ab_2\cdots ab_{k-1}a)$ for all $b_1,\dots,b_{k-1}\in \mathcal{B}$. Moreover we assume that the variables are centered, $E(a)=0$, and the their common variance is given by $\eta(b):=E(aba)$ for all $b\in \mathcal{B}$. If $\{a_i\}_{i=1}^{\infty}$ are free, Boolean or monotone independent, then as $N\to \infty$, the elements
\[
S_N=\frac{1}{\sqrt{N}}(a_1+a_2+\cdots+a_N)
\]
converge in distribution to an element $c$, whose respective operator-valued free, Boolean or monotone cumulants $\kappa\in\{r,b,h\}$ are all zero $\kappa_n(cb_1,\dots,cb_n,c)=0$, unless $n=2$ in which case we have 
\[
\kappa_2(cb,c)=\eta(b).
\]
\end{thm}

\begin{rem}
\label{rem.CLT.Banach}
The proof of OV CLT do not requires the use of positivity, hence Theorem \ref{OVCLT} still holds if we assume that $\mathcal{A}$ is just a Banach algebra rather than a $C^*$-algebra.
\end{rem}

Although the result is very similar for the three types of independence, the specific law of $c$ in each case is determined by the moment cumulant formula and the variance $\eta$. We recollect this in the following notation and definition.

\begin{nota}
We denote by $\mathcal{{NC}}_2(k)$ the set of non-crossing pair partitions of $[k]$. Let a partition $\pi \in \mathcal{NC}_2(k)$ and let $\eta:\mathcal{B}\to\mathcal{B}$ be a function which can be obtained as a variance of some variable $a$, namely $\eta(b)=E(aba)$, the map $\eta_\pi:\mathcal{B}^{n-1}\to\mathcal{B}$ is defined as $$\eta_{\pi}(b_1,b_2,\cdots,b_{k-1}):=E_{\pi}(ab_1,ab_2,\cdots,ab_{k-1},a).$$ 
Notice that since $\pi$ is a pairing, by definition $E_\pi$ breaks the product into pairs of the form $E(aba)$, where $\pi$ just indicates the way in which the pairs are nested. 
\end{nota}

\begin{defn}
\label{defn.OV.central.laws}
Let $\mathcal{B}\langle X \rangle$ denote the free algebra generated by an indeterminate variable $X$ over $\mathcal{B}$. The \emph{distribution} or \emph{law}, $\nu$, of a variable $c$, is the linear map $\nu:\mathcal{B}\langle X \rangle\to \mathcal{B}$ completely determined by 
\[ 
\nu(b_0Xb_1Xb_2\cdots b_{k-1}Xb_k)= E(b_0cb_1cb_2\cdots b_{k-1}cb_k).
\]
We are going to denote by $\nu_\eta$ the law of the central limit $c$ with variance $\eta$ from Theorem \ref{OVCLT}. We know that it is always symmetric, namely $\nu_{\eta}(b_0Xb_1Xb_2\cdots b_{k-1}Xb_k) = 0$ when $k$ is odd, and the even moments for each independence are the following.
\begin{itemize}
    \item [(1)] Free case: $c$ has the centered \emph{$\mathcal{B}$-valued semi-circle law} with variance $\eta$, whose even moments are
    \begin{equation*}
        \nu_{\eta}(b_0Xb_1Xb_2\cdots b_{k-1}Xb_k) =\sum\limits_{\pi\in\mathcal{NC}_2(k)}b_0\eta_{\pi}(b_1,\cdots,b_{k-1})b_k, \qquad \text{ for all even } k\in\mathbb{N}.
       \end{equation*} 
    \item [(2)] Boolean case: $c$ has the centered \emph{$\mathcal{B}$-valued Bernoulli law} with variance $\eta$, whose even moments are 
    \begin{equation*}
        \nu_{\eta}(b_0Xb_1Xb_2\cdots b_{k-1}Xb_k) =  b_0\eta(b_1)b_2\eta(b_2)\cdots \eta(b_{k-1})b_k, \qquad  \text{ for all even } k\in\mathbb{N}.
    \end{equation*}
    \item [(3)] Monotone case: $c$ has the centered \emph{$\mathcal{B}$-valued arcsine law} with variance $\eta$, whose even moments are 
    \begin{equation*}
        \nu_{\eta}(b_0Xb_1Xb_2\cdots b_{k-1}Xb_k) =\sum\limits_{\pi\in\mathcal{NC}_2(k)}\frac{1}{\tau(\pi)!}b_0\eta_{\pi}(b_1,\cdots,b_{k-1})b_k,  \qquad \text{ for all even } k\in\mathbb{N}.
    \end{equation*}
\end{itemize}

\end{defn}

\subsection{Operator-Valued Infinitesimal Probability Spaces} 

$\ $

\noindent In this section we extend the concept of operator-valued probability space to the infinitesimal setting, we will use the abbreviation OVI to refer to the operator-valued infinitesimal setting.

\begin{defn}
 We call $(\mathcal{A},\mathcal{B},E,E')$ an \emph{operator-valued infinitesimal probability space} if $(\mathcal{A},\mathcal{B},E)$ is an operator-valued probability space, and $E':\mathcal{A}\rightarrow \mathcal{B}$ is a linear map such that 
$$
E'(1)=0 \text{ and } E'(bab')=bE'(a)b'
$$
for all $b,b'\in\mathcal{B}$ and $a\in \mathcal{A}$. 

$(\mathcal{A},\mathcal{B},E,E')$ is further called a \emph{$C^*$-operator-valued infinitesimal probability space} if $(\mathcal{A},\mathcal{B},E)$ is a $C^*$-operator-valued probability space and $E':\mathcal{A}\rightarrow \mathcal{B}$ is a linear, $\mathcal{B}$-bimodule, selfadjoint map that is completely bounded. 
\end{defn}
\begin{defn}
For a given $C^*$-operator-valued infinitesimal probability space $(\mathcal{A},\mathcal{B},E,E')$ and $x=x^*\in \mathcal{A}$, the \emph{infinitesimal Cauchy transform} of $x$ is defined by
$$
g_x(b)=E'[(x-b)^{-1}]
$$
for all $b\in\mathcal{B}$ such that $b-x$ is invertible. 
\end{defn}

Then, we can extend the concept of freeness with amalgamation to the OVI setting.

\begin{defn}
\label{def.ovips}
Suppose that $(\mathcal{A},\mathcal{B},E,E')$ is an OVI probability space, unital subalgebras $(\mathcal{A}_i)_{i\in I}$ of $\mathcal{A}$ that contains $\mathcal{B}$ are called \emph{infinitesimally free} over $\mathcal{B}$ if for $n\in \mathbb{N}$ and $i_1,\dots, i_n\in I$, $i_1\neq \cdots \neq i_n$ and $a_j\in\mathcal{A}_{i_j}$ with $E(a_j)=0$ for all $j=1,\dots, n$, the following two conditions hold: 
\begin{eqnarray*}
E(a_1\cdots a_n)&=& 0\ ; \\
E'(a_1\cdots a_n) &=& \sum\limits_{j=1}^n E(a_1\cdots a_{j-1}E'(a_j)a_{j+1}\cdots a_n). 
\end{eqnarray*}
\end{defn}

To handle the OVI setting, given a unital algebra $\mathcal{A}$, we will constantly use its associated \emph{upper triangular algebra}
\begin{equation}
\label{eq.upper.triangular.algebra}
  \widetilde{\mathcal{A}} := \left \{ 
\begin{bmatrix} 
a  &  a' \\ 
0  &  a 
\end{bmatrix} 
\Bigg | a,a'\in\mathcal{A} \right \} .  
\end{equation}

Specially, given a OVI probability space $(\mathcal{A},\mathcal{B},E,E')$, we follow \cite{tseng2019operator} and define its corresponding \emph{upper triangular probability space}  $(\widetilde{\mathcal{A}},\widetilde{\mathcal{B}},\widetilde{E})$, where $\widetilde{E}:\widetilde{\mathcal{A}}\to \widetilde{\mathcal{B}}$ is given by
$$
\widetilde{E}
\Bigg( 
\begin{bmatrix} 
a  &  a' \\ 
0  &  a 
\end{bmatrix} 
\Bigg )
=
\begin{bmatrix} 
E(a)  &  E(a')+E'(a) \\ 
0  &  E(a) 
\end{bmatrix}. 
$$

As proved in \cite{tseng2019operator}, unital sub-algebras $(\mathcal{A}_i)_{i\in I}$ are infinitesimally free over $\mathcal{B}$ if and only if $(\widetilde{\mathcal{A}}_i)_{i\in I}$ are free over $\widetilde{\mathcal{B}}$, where $\widetilde{\mathcal{B}}$ and $\widetilde{\mathcal{A}}_i$ are defined as in $\eqref{eq.upper.triangular.algebra}$. Therefore, there is a connection between operator-valued infinitesimal freeness and freeness with amalgamation over an algebra of $2\times 2$ upper triangular matrices. This connection provides us a way to construct the operator-valued infinitesimal free convolution (see \cite{tseng2019operator}). This inspires us to construct the infinitesimal Boolean and Monotone convolutions by applying the analogue method. 


\subsection{Operator-valued infinitesimal cumulants}

$\ $

\noindent In order to define infinitesimal cumulants in an operator-valued space, we need to introduce some notation. We will do this at the best possible generality as this will be very useful in the rest of the paper. The general idea is to follow a similar approach as used to define cumulants.

\begin{nota}
\label{notation.partial}
Let $\mathcal{B}$ be a unital subalgebra of $\mathcal{A}$ (a unital algebra). Let $\{f_n:\mathcal{A}^n\to \mathcal{B}\}_{n\geq 1}$ and $\{\partial f_n:\mathcal{A}^n\to \mathcal{B}\}_{n\geq 1}$ be two sequences of multilinear maps. Given a partition $\pi \in NC(n)$ and a fixed block $V\in \pi$, we are going to denote by $\partial f_{\pi,V}:\mathcal{A}^n\to \mathcal{B}$ the map that is equal to $f_{\pi}$  (as in Definition \ref{partitions.multiplicative}) except that for the block $V$ we replace $f_{|V|}$ by $\partial f_{|V|}$. More specifically, we can define $\partial f_{\pi,V}$ recursively as follows:
\begin{enumerate}
    \item For $\pi=1_n$ and $V=[n]$ (the only possible choice of block) we set $\partial f_{\pi,V}=\partial f_n$.
    \item For $\pi\in NC(n)$ and $V\in \pi$ we pick $W=\{l+1,\dots, l+k\}\in \pi$ a interval block of $\pi$ and take $\pi_0=\pi\backslash W\in NC(n-k)$ the partition obtained by deleting from $\pi$ the block $W$. We have 2 cases,
    \begin{itemize}
        \item If $W=V$, then 
        $\partial f_{\pi,V}(x_1,\dots.x_n)= f_{\pi_0}(x_1,\dots,x_l \partial f_k(x_{l+1},\dots,x_{l+k}),x_{l+k+1},\dots,x_n).$
        \item If $W\neq V$, then 
        $\partial f_{\pi,V}(x_1,\dots.x_n)= \partial f_{\pi_0,V}(x_1,\dots,x_l f_k(x_{l+1},\dots,x_{l+k}),x_{l+k+1},\dots,x_n).$
    \end{itemize}
\end{enumerate}

As in Definition \ref{partitions.multiplicative}, it can be seen that $f_{\pi,V}$ does not depend on how we pick the interval block $V$. 

Then, we denote by $\partial f_{\pi}:\mathcal{A}^n\to \mathcal{B}$ the map such that
$$
\partial f_{\pi}(a_1,\dots,a_n):=\sum\limits_{V\in \pi}\partial f_{\pi,V}(a_1,\dots,a_n).
$$
\end{nota}

When working with families of functions in OVI probability spaces its useful to merge functions into a single function in the associated upper triangular probability space.

\begin{nota}
\label{nota.tilde.cumulant}
Let $\mathcal{B}$ be a unital subalgebra of $\mathcal{A}$ (a unital algebra), and let $\{f_n:\mathcal{A}^n\to \mathcal{B}\}_{n\geq 1}$ and $\{\partial f_n:\mathcal{A}^n\to \mathcal{B}\}_{n\geq 1}$ be two sequences of multilinear maps. We merge these sequences into a sequence of multilinear maps $\widetilde{f}_n: \widetilde{\mathcal{A}}^n\to \widetilde{\mathcal{B}}$ such that 
\begin{equation*}
\widetilde{f}_n
\Bigg( 
\begin{bmatrix} 
a_1  &  a'_1 \\ 
0  &  a_1
\end{bmatrix} , \dots, \begin{bmatrix} 
a_n  &  a'_n \\ 
0  &  a_n 
\end{bmatrix} 
\Bigg )
:=
\begin{bmatrix} 
f_n(a_1,\dots,a_n)  &  \sum\limits_{j=1}^n f_n(a_1,\dots, a_{j-1},a'_{j},\dots, a_n)+\partial f_n(a_1,\dots,a_n) \\ 
0  &  f_n(a_1,\dots,a_n) 
\end{bmatrix}.
\end{equation*}
Since $\{\widetilde{f}_n:\widetilde{\mathcal{A}}^n\to \widetilde{\mathcal{B}}\}_{n\geq 1}$ is a family of multilinear maps we can define $\widetilde{f}_\pi:=(\widetilde{f})_\pi$ following Definition \ref{partitions.multiplicative}.
\end{nota}

\begin{rem}
\label{rem.Eut}
To clarify the notation that we just introduce, we gave an example considering the case $f=E$, by this we mean taking $f_n(a_1,\dots,a_n)=E(a_1\dots a_n)$ and $\partial f_n(a_1,\dots,a_n)=E'(a_1\dots a_n)$ for all $n\in \mathbb{N}$ and $a_1\dots,a_n\in\mathcal{A}$.  Then, previous notation is telling us that
$$
\partial E_{\pi}(a_1,\dots,a_n):=\sum\limits_{V\in \pi}\partial E_{\pi,V}(a_1,\dots,a_n),
$$
and we get that $\widetilde{f}_n(a_1,\dots,a_n)=\widetilde{E}(a_1\dots a_n)$ as we would expect.
\end{rem}

We now state a lemma in connection to the previous notation. This is a technical but very useful result, that allows to clearly parallel the OVI setting with the OV setting. A special case of this lemma appears in the proof of Lemma 10 in \cite{tseng2019operator}.
\begin{lem}
\label{lemma10.Tseng}
Let $\mathcal{B}$ be a unital subalgebra of $\mathcal{A}$, and let $\{f_n:\mathcal{A}^n\to \mathcal{B}\}_{n\geq 1}$ and $\{\partial f_n:\mathcal{A}^n\to \mathcal{B}\}_{n\geq 1}$ be sequences of multilinear maps. Consider $A_1,\dots,A_n\in \widetilde{\mathcal{A}}$  with 
$
A_i=\begin{bmatrix}
a_i & a'_i \\
0  & a_i
\end{bmatrix}
$ for each $i=1,\dots, n$. Then for each $n\in \mathbb{N}$ and $\pi\in NC(n)$, we have 
\begin{equation}\label{eqn.lemma10.Tseng}
\widetilde{f}_{\pi}(A_1,\dots ,A_n) =   
\begin{bmatrix}
f_{\pi}(a_1,\dots ,a_n) & \sum\limits_{j=1}^n f_{\pi}(a_1,\dots, a_{j-1},a'_{j},a_{j+1},\dots, a_n)+\partial f_{\pi}(a_1,\dots, a_n)\\
0 & f_{\pi}(a_1,\dots, a_n)
\end{bmatrix}.
\end{equation}
\end{lem}

\begin{proof}
We proceed by induction on the size $|\pi|$. For the base case, if $|\pi|=1$ then $\pi=1_n$ and \eqref{lemma10.Tseng} clearly holds from the definition of $\widetilde{f}_{1_n}=\widetilde{f}_{n}$. For $|\pi|>1$, consider an interval block $W=\{k+1,\dots,k+l\}$ of $\pi$ and $\pi_0=\pi\backslash V$. By definition:
\begin{equation}\label{eqn.1.proof.lemma10.Tseng}
\widetilde{f}_{\pi}(A_1,\dots ,A_n) = \widetilde{f}_{\pi_0}(A_1,\dots,A_{k-1},A_k\widetilde{f}_{l}(A_{k+1},\dots ,A_{k+l}),A_{k+l+1},\dots ,A_n).
\end{equation}
Using the definition, the diagonal entries of $A_k\widetilde{f}_{\pi}(A_{k+1},\dots ,A_{k+l})$ are $a_kf_{\pi}(a_{k+1},\dots ,a_{k+l})$, and the $(2,1)$-entry is
\begin{equation*}
K:=a_k\sum\limits_{j=k+1}^{k+l} f_{l}(a_k,\dots,a'_{j},\dots, a_{k+l})+a_k\partial f_{l}(a_{k+1},\dots, a_{k+l})+ a'_kf_{\pi}(a_{k+1},\dots ,a_{k+l}).
\end{equation*}
Using this equation and the induction hypothesis on $\pi_0$ the diagonal entries of the right hand side of \eqref{eqn.1.proof.lemma10.Tseng} become
$$\widetilde{f}_{\pi_0}(a_1,\dots,a_k\widetilde{f}_{l}(a_{k+1},\dots ,a_{k+l}),\dots ,a_n)=f_{\pi}(a_1,\dots ,a_n),$$
as desired. On the other hand, abbreviating $a_W:=(a_{k+1},\dots,a_{k+l})$, we get that the $(2,1)$-entry of \eqref{eqn.1.proof.lemma10.Tseng} is
\begin{eqnarray*}
 &&\sum\limits_{j=1}^{k-1} f_{\pi_0}(a_1,\dots,a'_{j},\dots,a_kf_l(a_W),\dots, a_n)+ f_{\pi_0}(a_1,\dots, a_{k-1},K,a_{k+1},\dots, a_n)\\
 &&+\sum\limits_{j=k+1}^n f_{\pi}(a_1,\dots,a_kf_l(a_W),\dots,a'_{j},\dots, a_n)+\partial f_{\pi_0}(a_1,\dots,a_kf_l(a_W),\dots, a_n)\\
 &&= \sum\limits_{j=1}^{n} f_{\pi_0}(a_1,\dots,a'_{j},\dots,a_kf_l(a_{k+1},\dots,a_{k+l}),\dots, a_n) \\
 &&\quad +f_{\pi_0}(a_1,\dots,a_k\partial f_l(a_W),\dots, a_n) +\partial f_{\pi_0}(a_1,\dots,a_kf_l(a_W),\dots, a_n) \\
 &&=\sum\limits_{j=1}^n f_{\pi}(a_1,\dots, a_{j-1},a'_{j},a_{j+1},\dots, a_n)+\partial f_{\pi}(a_1,\dots, a_n),
\end{eqnarray*}
as desired.
\end{proof}


Equation \eqref{eqn.lemma10.Tseng}, when $f=E$, plays a key role in the characterization of infinitesimal freeness in terms of vanishing of mixed OVI cumulants. 

\begin{defn}[OVI free cumulants]
For $n\in \mathbb{N}$, we define $\partial r^{\mathcal{B}}_{n}:\mathcal{A}^n\rightarrow \mathcal{B}$ to be the multilinear map determined by the formula
\[
\partial r^{\mathcal{B}}_n(a_1,a_2,\ldots, a_n)=\sum\limits_{\pi \in NC(n)}Mob(\pi,1_n)\partial E_{\pi}(a_1,a_2,\ldots,a_n), \qquad \forall a_1,a_2,\ldots,a_n\in\mathcal{A},
\]
The sequence $(\partial r^{\mathcal{B}}_n)_n$ is called \emph{operator-valued infinitesimal free cumulants} of $(\mathcal{A},\mathcal{B},E,E').$
\end{defn}

\begin{thm} 
Given an operator-valued infinitesimal probability space $(\mathcal{A}, \mathcal{B}, E, E')$, and subalgebras 
$\mathcal{A}_1,\dots, \mathcal{A}_n$ of $\mathcal{A}$. Then the following two statements are equivalent: 
\begin{itemize}
    \item [(1)] $\mathcal{A}_1,\dots, \mathcal{A}_n$ are infinitesimally free with respect to $(E, E').$ 
    \item [(2)] For every $n\geq 2$ and $i_1,\dots, i_s\in [n]$ which are not all equal, and for $a_1\in \mathcal{A}_{i_1},\dots, a_s\in\mathcal{A}_{i_s}$, we have $r^{\mathcal{B}}_s(a_1,\dots,a_s)=\partial r^{\mathcal{B}}_s(a_1,\dots,a_s)=0$.  
\end{itemize}
\end{thm}

\section{OVI Boolean Independence}
\label{sec:Boolean.Indep}

In this section we extend the notion of OV Boolean independence to the infinitesimal setting. Once we have this, we stablish the connection between OVI Boolean independence and operator-valued Boolean independence over an algebra of $2\times 2$ upper triangular matrices. This allow us to prove Theorem \ref{Thm.Boolean.add.conv} describing the OVI additive Boolean convolution in terms of the infinitesimal Cauchy transform and the $F$-transform. Then we introduce the notion of OVI Boolean cumulants, study its basic notions and use them to link the OVI setting with the scalar-valued infinitesimal setting.

\subsection{OVI Boolean Convolution}

$\ $

\noindent The notion of OV Boolean independence can be extended to the infinitesimal setting as follows. 

\begin{defn}
Suppose that $(\mathcal{A},\mathcal{B},E,E')$ is an OVI probability space. Subalgebras $(\mathcal{A}_i)_{i\in I}$ containing $\mathcal{B}$ are called \emph{infinitesimally Boolean independent} if for all $n\in\mathbb{N}$ and $a_1,\dots, a_n\in\mathcal{A}$ such that $a_j\in\mathcal{A}_{i_j}$ where $i_1\neq i_2\neq \dots \neq i_n\in I$, we have
\begin{eqnarray}
E(a_1\cdots a_n) &=& E(a_1)\cdots E(a_n) ; \label{eq.Boolean.indep}\\
E'(a_1\cdots a_n) &=&\sum_{j=1}^n E(a_1)\cdots E(a_{j-1})E'(a_j)E(a_{j+1})\cdots E(a_{n}). \label{eq.Boolean.inf.indep}
\end{eqnarray}
\end{defn}

OVI Boolean independence can be nicely phrased if we consider the associated upper triangular probability space.

\begin{prop}  \label{p1}
Suppose that $(\mathcal{A},\mathcal{B},E,E')$ is an OVI probability space, and $(\widetilde{\mathcal{A}},\widetilde{\mathcal{B}},\widetilde{E})$ be the corresponding upper triangular probability space.
Then subalgebras $(\mathcal{A}_i)_{i\in I}$ containing $\mathcal{B}$ are infinitesimally Boolean independent over $\mathcal{B}$ if and only if $(\widetilde{\mathcal{A}}_i)_{i\in I}$ are Boolean independent over $\widetilde{\mathcal{B}}$.
\end{prop}

\begin{proof}
Assume that $(\mathcal{A}_i)_{i\in I}$ are infinitesimally Boolean independent. Let $A_1,\dots, A_n\in \widetilde{\mathcal{A}}$ with each 
$A_j=
\begin{bmatrix}
a_j & a_j'\\
0 & a_j
\end{bmatrix} \in \widetilde{\mathcal{A}}_{i_j}$  
where $i_1\neq i_2\neq \cdots \neq i_n$.
Note that
\begin{equation*}
\widetilde{E}( A_1\cdots A_n ) 
= \begin{bmatrix}
E(a_1\cdots a_n) & E'(a_1\cdots a_n) + \sum\limits_{j=1}^n E( a_1\cdots a_{j-1}a'_{j}a_{j+1}\cdots a_n) \\
0 & E(a_1\cdots a_n)
\end{bmatrix}.
\end{equation*}

By our assumption, the latter is equal to 
\begin{eqnarray*}
& &
\begin{bmatrix}
E(a_1)\cdots E(a_n) & \sum\limits_{j=1}^n E( a_1)\cdots E(a_{j-1})\big( E'(a_j)+E(a'_{j}) \big)E(a_{j+1})\cdots E(a_n) \\
0 & E(a_1)\cdots E(a_n) 
\end{bmatrix} \\
&=&
\begin{bmatrix}
E(a_1) & E'(a_1)+E(a'_1) \\
0 & E(a_1)
\end{bmatrix}
\cdots 
\begin{bmatrix}
E(a_n) & E'(a_n)+E(a_n') \\
0 & E(a_n)
\end{bmatrix} \\
&=& \widetilde{E}(A_1)\cdots \widetilde{E}(A_n).
\end{eqnarray*}
Therefore, $(\widetilde{\mathcal{A}}_i)_{i\in I}$ are Boolean independent over $\widetilde{\mathcal{B}}$.\\

Conversely, suppose that $(\widetilde{\mathcal{A}}_i)_{i\in I}$ are Boolean independent over $\widetilde{\mathcal{B}}$. Let $a_1\cdots,a_n\in\mathcal{A}$ such that $a_j\in\mathcal{A}_{i_j}$ with $i_1\neq \cdots \neq i_n$. For each $j$, we consider 
$A_j=\begin{bmatrix}
a_j & 0 \\
0 & a_j
\end{bmatrix}$.
Since $(\widetilde{\mathcal{A}}_{i})_{i\in I}$ are Boolean independent, we have $\widetilde{E}(A_1\cdots A_n) =\widetilde{E}(A_1)\cdots \widetilde{E}(A_n),$ which implies that
\[
\begin{bmatrix}
E(a_1\cdots a_n) & E'(a_1\cdots a_n) \\
0 & E(a_1\cdots a_n)
\end{bmatrix} = 
\begin{bmatrix}
E(a_1)\cdots E(a_n) & \sum\limits_{j=1}^n E(a_1)\cdots E(a_{j-1})E'(a_j)E(a_{j+1})\cdots E(a_n) \\
0 & E(a_1)\cdots E(a_n)
\end{bmatrix}.
\] Thus, reading the two upper entries we conclude that
\begin{eqnarray*}
E(a_1\cdots a_n) &=& E(a_1)\cdots E(a_n) ; \\
E'(a_1\cdots a_n) &=& \sum\limits_{j=1}^n E(a_1)\cdots E(a_{j-1})E'(a_j)E(a_{j+1})\cdots E(a_{n}).
\end{eqnarray*}
\end{proof}

We are now in position to describe the OVI additive Boolean convolution.

\begin{proof}[Proof of theorem \ref{Thm.Boolean.add.conv}]
Suppose that $(\mathcal{A},\mathcal{B},E,E')$ is a $C^*$-operator-valued infinitesimal probability space, and if $x,y\in\mathcal{A}$ are two selfadjoint random variables that are infinitesimally Boolean independent. Let $X=
\begin{bmatrix}
x & 0 \\
0 & x
\end{bmatrix}$ and $Y=
\begin{bmatrix}
y & 0 \\
0 & y
\end{bmatrix}$. Then by Proposition \ref{p1}, we have $X, Y$ are Boolean independent over $\widetilde{\mathcal{B}}$. 
By \eqref{Beqn}, we have
$$
B_X \Bigg (\begin{bmatrix}
b & c \\
0 & b
\end{bmatrix}\Bigg )+B_Y \Bigg (\begin{bmatrix}
b & c \\
0 & b
\end{bmatrix} \Bigg ) = B_{X+Y} \Bigg (\begin{bmatrix}
b & c \\
0 & b
\end{bmatrix} \Bigg).
$$
That is, 
\begin{equation}\label{beqn3}
\begin{bmatrix}
b & c \\
0 & b
\end{bmatrix} - F_{X}\Bigg(\begin{bmatrix}
b & c \\
0 & b
\end{bmatrix}\Bigg) 
+ 
\begin{bmatrix}
b & c \\
0 & b
\end{bmatrix} - F_{Y}\Bigg(\begin{bmatrix}
b & c \\
0 & b
\end{bmatrix}\Bigg) 
=\begin{bmatrix}
b & c \\
0 & b
\end{bmatrix} - F_{X+Y}\Bigg(\begin{bmatrix}
b & c \\
0 & b
\end{bmatrix}\Bigg).
\end{equation}
We note that 
\begin{eqnarray*}
F_{X}\Bigg(\begin{bmatrix}
b & c \\
0 & b
\end{bmatrix}\Bigg)  &=&
G_{X}\Bigg(\begin{bmatrix}
b & c \\
0 & b
\end{bmatrix}\Bigg)^{-1} 
=
\begin{bmatrix}
G_x(b) & G_x'(b)c+g_x(b) \\
0 & G_x(b)    
\end{bmatrix}^{-1}   \\
&=&
\begin{bmatrix}
F_x(b) & -F_x(b)(G_x'(b)c+g_x(b))F_x(b) \\
0 & F_x(b)
\end{bmatrix}   
= 
\begin{bmatrix}
F_x(b) & F_x'(b)c-F_x(b)g_x(b)F_x(b) \\
0 & F_x(b)
\end{bmatrix}.
\end{eqnarray*}
Thus, comparing the $(1,2)$-entry of \eqref{beqn3}, we obtain
\begin{eqnarray*}
F_x(b)g_x(b)F_x(b) + F_y(b)g_y(b)F_y(b) &=& F_{x+y}(b)g_{x+y}(b)F_{x+y}(b) \\
                                            &=& (F_x(b)+F_y(b)-b)g_{x+y}(b)(F_x(b)+F_y(b)-b). 
\end{eqnarray*}
Hence, we conclude that
$$
g_{x+y}(b) = (F_x(b)+F_y(b)-b)^{-1} \Big( F_x(b)g_x(b)F_x(b) + F_y(b)g_y(b)F_y(b) \Big) (F_x(b)+F_y(b)-b)^{-1}. 
$$
\end{proof}
\begin{rem}
When $\mathcal{B}=\mathbb{C}$, we obtain the scalar version of infinitesimal Boolean convolution. To be precise, if $x=x^*$ and $y=y^*$ are infinitesimally free from a $C^*$-infinitesimal probability space $(\mathcal{A},\varphi,\varphi')$, we have
$$
g_{x+y}(z)=\frac{g_x(z)F_x(z)^2+g_y(z)F_y(z)^2}{(F_x(z)+F_y(z)-z)^2}.
$$
\end{rem}

\subsection{OVI Boolean cumulants}

$\ $

\noindent The notion of OVI Boolean cumulants is a direct analogue of the OV Boolean cumulants.
\begin{defn}
Suppose that $(\mathcal{A},\mathcal{B},E,E')$ is an OVI probability space, with operator-valued Boolean cumulants $\{\beta_n^{\mathcal{B}}:\mathcal{A}^n\to \mathcal{B}\}$ (see equation \eqref{boo-mom}). We define the \emph{operator-valued infinitesimally Boolean cumulants} $\{\partial\beta_n^{\mathcal{B}}:\mathcal{A}^n\to \mathcal{B}\}$ to be the (unique) family of functions such that for all $n\in\mathbb{N}$ and $a_1,\dots, a_n\in\mathcal{A}$ 
\begin{equation}
\label{inf.boo-mom}
\partial\beta^{\mathcal{B}}_n(a_1,\cdots,a_n) = \sum\limits_{\pi\in \mathcal{I}(n)}(-1)^{|\pi|-1}\partial E_{\pi} (a_1,\cdots,a_n).
\end{equation}
\end{defn}

A remarkable thing is that the upper triangular world is again useful here.

\begin{rem}
Consider the corresponding upper triangular operator-valued probability space $(\widetilde{\mathcal{A}},\widetilde{\mathcal{B}},\widetilde{E})$. Then the operator-valued Boolean cumulants $\{\beta^{\widetilde{\mathcal{B}}}_n:\widetilde{\mathcal{A}}^n\to \widetilde{\mathcal{B}}\}$, are intimately related to the operator-valued and OVI Boolean cumulants of $(\mathcal{A},\mathcal{B},E,E')$. Indeed, we consider $A_j=
\begin{bmatrix}
a_j & a_j'\\
0 & a_j
\end{bmatrix} \in \widetilde{\mathcal{A}}$ for $j=1,\dots,n$ and use Equation \eqref{boo-mom} and Lemma \eqref{lemma10.Tseng} to obtain
\begin{align*}
\beta^{\tilde{\mathcal{B}}}_n(A_1,\dots,A_n) &= \sum\limits_{\pi\in \mathcal{I}(n)}(-1)^{|\pi|-1} \widetilde{E}_\pi(A_1,\cdots, A_n)\\
&=\sum\limits_{\pi\in \mathcal{I}(n)}(-1)^{|\pi|-1}
\begin{bmatrix}
E_{\pi}(a_1,\dots ,a_n) & \partial E_{\pi}(a_1,\dots, a_n)+\sum\limits_{j=1}^n E_{\pi}(a_1,\dots, a_{j-1},a'_{j},\dots, a_n)\\
0 & E_{\pi}(a_1,\dots, a_n)
\end{bmatrix}\\
&=
\begin{bmatrix}
\beta^{\mathcal{B}}_n(a_1,\dots,a_n) &  \partial \beta^{\mathcal{B}}_n(a_1,\dots,a_n)  + \sum\limits_{j=1}^n \beta^{\mathcal{B}}_n( a_1,\cdots a_{j-1},a'_{j},a_{j+1},\cdots, a_n) \\
0 & \beta^{\mathcal{B}}_n(a_1,\dots,a_n) 
 \end{bmatrix}
\\ &= \widetilde{\beta}_n(A_1,\dots,A_n).
\end{align*}
Therefore $\beta^{\tilde{\mathcal{B}}}_n=\widetilde{\beta}_n$, where $\widetilde{\beta}^{\tilde{\mathcal{B}}}_n$ are the OV Boolean cumulants of $(\widetilde{\mathcal{A}},\widetilde{\mathcal{B}},\widetilde{E})$, while $\widetilde{\beta}_n$ are the multilinear functional obtained from merging the OV and OVI cumulants of $(\mathcal{A},\mathcal{B},E,E')$ as in Notation \ref{nota.tilde.cumulant}. From now on we will use $\beta^{\tilde{\mathcal{B}}}_n$ or $\widetilde{\beta}_n$ indistinctly to refer to Boolean cumulants of the associated upper triangular probability space.
\end{rem}


\begin{rem}
We are also able to give a formula writing infinitesimal moments in terms of OVI Boolean cumulants, this could be seen as an inverse of formula \eqref{inf.boo-mom}:
\begin{equation}
\label{inf.mom-boo}
E'_n(a_1,\dots,a_n) = \sum\limits_{\pi\in \mathcal{I}(n)}\partial \beta_{\pi}^{\mathcal{B}} (a_1,\dots,a_n).
\end{equation}

Note that it is not possible to get this formula just by performing Möbius inversion to \eqref{inf.boo-mom}. However, we can make use of the relation to the upper triangular OV space and perfom the Möbius inversion there.
Indeed,  Equation \eqref{mom-boo}, the previous remark and Equation \eqref{eqn.lemma10.Tseng} (with $f= \beta$) yield
\begin{align*}
\widetilde{E}_n(A_1,\dots,A_n)) &= \sum\limits_{\pi\in \mathcal{I}(n)} \widetilde{\beta}_\pi(A_1,\cdots, A_n)\\
&=\sum\limits_{\pi\in \mathcal{I}(n)}
\begin{bmatrix}
\beta_{\pi}^{\mathcal{B}} (a_1,\dots ,a_n) & \partial\beta_{\pi}^{\mathcal{B}} (a_1,\dots, a_n)+\sum\limits_{j=1}^n \beta_{\pi}^{\mathcal{B}} (a_1,\dots, a_{j-1},a'_{j},a_{j+1},\dots, a_n)\\
0 & \beta_{\pi}^{\mathcal{B}} (a_1,\dots, a_n)
\end{bmatrix}
\end{align*}

If we set $a_1'=a_2'=\dots=a_n'=0$ and just focus on the $(1,2)$-entry, we obtain \eqref{inf.mom-boo}
\end{rem}

By a similar argument to Theorem 10 of \cite{tseng2019operator}, we can obtain the following result.
\begin{thm}\label{thm3}
Suppose that $(\mathcal{A}, \mathcal{B}, \mathbb{E},\mathbb{E'})$ is an operator-valued infinitesimal probability space, and 
$\mathcal{A}_1,\dots, \mathcal{A}_n$ are unital subalgebras of $\mathcal{A}$ that contain $\mathcal{B}$. Then the following two statements are equivalent: 
\begin{itemize}
    \item [(1)] $\mathcal{A}_1,\dots, \mathcal{A}_n$ are infinitesimally Boolean independent with respect to $(\mathbb{E},\mathbb{E'}).$ \\
    \item [(2)] For every $n\geq 2$ and $i_1,\dots, i_s\in [n]$ which are not all equal, and for $a_1\in \mathcal{A}_{i_1},\dots, a_s\in\mathcal{A}_{i_s}$, we have $\beta^{\mathcal{B}}_s(a_1,\dots,a_s)=\partial\beta^{\mathcal{B}}_s(a_1,\dots,a_s)=0$.  
\end{itemize}
\end{thm}

\begin{rem}
The scalar version of Theorem \ref{thm3} specializes to Theorem 3.7 of \cite{Has}, and it can be applied to test if random matrix models which are asymptotically Boolean independent are also infinitesimally Boolean. In the Appendix we study two of these models. The first model is from the Section 5 of \cite{MLE}, and the second model comes from Section 7 of \cite{lenczewski2014limit}. It turns out that even though both models are asymptotically Boolean independent, they are not asymptotically infinitesimally Boolean independent. \end{rem}


Cumulants allow us to carry over scalar-valued infinitesimal Boolean independence into a OVI setting. We will prove now Theorem \ref{thm.relation.inf.boo.matrices} that asserts the following. Suppose that $(\mathcal{M},\varphi,\varphi')$ is an infinitesimal probability space. For a fixed $N\in\mathbb{N}$ define the corresponding OVI probability space $(\mathcal{A},\mathcal{B},E,E')$ where 
\begin{equation*}
\mathcal{A}=M_N(\mathcal{M}),\ \ \mathcal{B}=M_N(\mathbb{C}),\ \ E=\varphi\otimes Id_N,\ \ E'=\varphi'\otimes Id_N.     
\end{equation*}

Now, if the sets $\{a^{(1)}_{i,j}\colon1\le i,j\le N\},\dots, \{a^{(k)}_{i,j}\colon1\le i,j\le N\}$ are infinitesimally Boolean in $\mathcal{M}$, then the elements $A^{(1)}:=[a^{(1)}_{i,j}]_{i,j=1}^N,\dots, A^{(k)}:=[a^{(k)}_{i,j}]_{i,j=1}^N$ are infinitesimally Boolean in $M_N(\mathcal{M})$ over $M_N(\mathbb{C})$.

\begin{proof}[Proof of Theorem \ref{thm.relation.inf.boo.matrices}]
Notice that it suffices to show that for all $h\geq 2$,
\begin{equation}
\label{eq.bool.scalar.operator}
\beta^{\mathcal{B}}_h(A_{i_1},A_{i_2},\ldots,A_{i_h})=\partial\beta^{\mathcal{B}}_h(A_{i_1},A_{i_2},\ldots,A_{i_h})=0
\end{equation}
whenever $i_1,i_2,\ldots,i_h\in [k]$ are not all equal. Now, 
we recall that Boolean cumulants in $(\mathcal{A},\mathcal{B},E,E')$ can be written in terms of Boolean cumulants in $(\mathcal{M},\varphi,\varphi')$. Indeed, suppose that $X^{(k)}=[x_{i,j}^{(k)}]$ are elements in $\mathcal{A}=M_N(\mathcal{M})$ for all $k=1,\dots,n$. Then,
\begin{equation}\label{booeqn1}
[\beta^{\mathcal{B}}_n(X^{(1)},X^{(2)},\ldots ,X^{(n)})]_{i,j} = \sum\limits_{i_2,i_3,\ldots i_n=1}^{N}\beta_{n}(x_{i,i_2}^{(1)},x_{i_2,i_3}^{(2)}\ldots ,x_{i_n,j}^{(n)}).
\end{equation}
In addition, by \cite[Lemma 11]{tseng2019operator}, for $\pi \in \mathcal{I}(n)$ and $V\in \pi$, we have
\[
[\partial E_{\pi,V}(X^{(1)},X^{(2)},\ldots ,X^{(n)})]_{i,j}= \sum\limits_{i_2,i_3,\ldots,i_n=1}^{N}\partial\varphi_{\pi,V}(x_{i,i_2}^{(1)},x_{i_2,i_3}^{(2)},\ldots ,x_{i_n,j}^{(n)}).
\]
Therefore, 
\begin{eqnarray*}
[\partial\beta^{\mathcal{B}}_n(X^{(1)},X^{(2)},\ldots ,X^{(n)})]_{i,j}
    &=&\sum\limits_{\pi\in \mathcal{I}(n)}(-1)^{|\pi|-1}\sum\limits_{V\in \pi}[\partial E_{\pi,V}(X^{(1)},X^{(2)},\ldots ,X^{(n)})]_{i,j} \\
    &=&\sum\limits_{i_2,i_3,\ldots,i_n=1}^{N}\sum\limits_{\pi\in \mathcal{I}(n)}(-1)^{|\pi|-1}\sum\limits_{V\in \pi}\partial\varphi_{\pi,V}(x_{i,i_2}^{(1)},x_{i_2,i_3}^{(2)},\ldots ,x_{i_n,j}^{(n)}) \\
    &=&\sum\limits_{i_2,i_3,\ldots i_n=1}^{N}\beta'_{n}(x_{i,i_2}^{(1)},x_{i_2,i_3}^{(2)}\ldots ,x_{i_n,j}^{(n)}).
\end{eqnarray*}

Thus, we have the analogue relation for infinitesimal Boolean cumulants:
\begin{equation} \label{booeqn2}
[\partial\beta^{\mathcal{B}}_n(X^{(1)},X^{(2)},\ldots ,X^{(n)})]_{i,j}= \sum\limits_{i_2,i_3,\ldots i_n=1}^{N}\beta'_{n}(x_{i,i_2}^{(1)},x_{i_2,i_3}^{(2)}\ldots ,x_{i_n,j}^{(n)}).
\end{equation}

Going back to our problem,  since  $\{a^{(1)}_{i,j}\colon 1\le i,j\le N\},\dots, \{a^{(k)}_{i,j}\colon 1\le i,j\le N\}$ are infinitesimal Boolean independent, equations \eqref{booeqn1} and \eqref{booeqn2} yield
$$
[\beta^{\mathcal{B}}_n(A_{i_1},A_{i_2},\ldots,A_{i_k})]_{i,j}= [\partial\beta^{\mathcal{B}}_n(A_{i_1},A_{i_2},\ldots,A_{i_k})]_{i,j} = 0, \qquad \forall 1\le i,j\le N.
$$
Therefore we obtain \eqref{eq.bool.scalar.operator} and the desired result follows.
\end{proof}

\section{OVI Monotone Independence}
\label{sec:Monotone.indep}

In this section we extend the notion of OV monotone independence to the infinitesimal setting. We study the OVI Boolean convolution using infinitesimal Cauchy transform. Then we introduce the notion of OVI monotone cumulants and discuss whether there is a relation between scalar-valued infinitesimal and OVI monotone independence.

\subsection{OVI Monotone Convolution}

$\ $

\noindent The notion of OV monotone independence can be naturally extended to the infinitesimal setting as follows.
\begin{defn}
Suppose that $(\mathcal{A},\mathcal{B},E,E')$ is an OVI probability space, and assume that $I$ is equipped with a linear order $<$. Subalgebras $(\mathcal{A}_i)_{i\in I}$ that contain $\mathcal{B}$ are called \emph{infinitesimally monotone over $\mathcal{B}$} if 
\begin{eqnarray}
E(a_1\cdots a_{j-1}a_j a_{j+1}\cdots a_n) &=& E(a_1\cdots a_{j-1}E(a_j)a_{j+1}\cdots a_n) ; \label{eqn.ov.monotone}\\
E'(a_1\cdots a_{j-1}a_j a_{j+1}\cdots a_n) &=&  E(a_1\cdots a_{j-1}E'(a_j)a_{j+1}\cdots a_n) +  E'(a_1\cdots a_{j-1}E(a_j)a_{j+1}\cdots a_n). \label{eqn.ov.inf.monotone}
\end{eqnarray}
whenever $a_j\in\mathcal{A}_{i_j}, i_j\in I$ for all $j$ and $ i_{j-1} < i_j > i_{j+1}$, where one of the inequalities is eliminated if $j=1$ or $j=n$.
\end{defn}

As in the free and Boolean cases, OVI monotone independence can be nicely phrased if we consider the associated upper triangular probability space.

\begin{prop}\label{p2}
Suppose that $(\mathcal{A},\mathcal{B},E,E')$ is an OVI probability space, subalgebras $(\mathcal{A}_i)_{i\in I}$ are infinitesimally monotone independent over $\mathcal{B}$ if and only if $(\widetilde{\mathcal{A}}_i)_{i\in I}$ are monotone independent over $\widetilde{\mathcal{B}}$.
\end{prop}
\begin{proof}
Suppose that $(\mathcal{A}_i)_{i\in I}$ are infinitesimally monotone independent over $\mathcal{B}$, and we let $A_1,\dots, A_n\in \widetilde{\mathcal{A}}$ with each 
$A_j=
\begin{bmatrix}
a_j & a_j'\\
0 & a_j
\end{bmatrix} \in \widetilde{\mathcal{A}}_{i_j}$ such that $i_1,\dots ,i_n\in I$ with $i_{j-1} < i_j > i_{j+1}$. We let $M=\widetilde{E}(A_1\cdots A_{j-1}\widetilde{E}(A_j)A_{j+1}\cdots A_n)$, and then we shall show that $\widetilde{E}(A_1\cdots A_n)=M.$ \\
Since $\widetilde{E}(A_j)=\begin{bmatrix}
E(a_j) & E'(a_j)+E(a_j')\\
0 & E(a_j)
\end{bmatrix}$, the diagonal entries $M_{11}=M_{22}$ of $M$ are equal to
$$M_{11}=E(a_1\cdots a_{j-1}E(a_j)a_{j+1}\cdots a_n)=E(a_1\cdots  a_n).$$

On the other hand, the $(1,2)$-entry of $M$ is

\begin{align}
M_{12}&= E'(a_1\cdots a_{j-1}E(a_j)a_{j+1}\cdots a_n)  + E\Big(a_1\cdots a_{j-1}\Big( E'(a_j)+E(a_j')\Big) a_{j+1}\cdots a_n\Big) \label{long.formula} \\ 
 &+\sum_{\substack{1\leq k\leq n \\  j\neq k}} E(a_1\cdots a_{j-1}E(a_j)a_{j+1}\cdots a_{k-1}a_{k}'a_{k+1}\cdots a_n)  \nonumber \\
 &=  E(a_1\cdots a_{j-1}E'(a_j)a_{j+1}\cdots a_n) +  E'(a_1\cdots a_{j-1}E(a_j)a_{j+1}\cdots a_n) \nonumber \\
 & +  \sum_{\substack{1\leq k\leq n \\  j\neq k}} E(a_1\cdots a_{j-1}E(a_j)a_{j+1}\cdots a_{k-1}a_k'a_{k+1}\cdots a_n) +
 E(a_1\cdots a_{j-1}E(a_j')a_{j+1}\cdots a_n) \nonumber \\
 &= E'(a_1\cdots a_n) + \sum\limits_{k=1}^n E(a_1\cdots a_{k-1}a'_ka_{k+1}\cdots a_n) \nonumber,
\end{align}
where the second equality is just a simple manipulation and in the last equality we used formulas \eqref{eqn.ov.monotone} and \eqref{eqn.ov.inf.monotone}. Therefore we conclude that

$$
\widetilde{E}( A_1\cdots A_n ) 
= \begin{bmatrix}
E(a_1\cdots a_n) & E'(a_1\cdots a_n) + \sum\limits_{k=1}^n E(a_1\cdots a_{k-1}a'_ka_{k+1}\cdots a_n) \\
0 & E(a_1\cdots a_n)
\end{bmatrix} = \begin{bmatrix}
M_{11} & M_{12} \\ 0 & M_{22}
\end{bmatrix} = M.
$$

Conversely, suppose that $(\widetilde{\mathcal{A}}_i)_{i\in I}$ are monotone independent over $\widetilde{\mathcal{B}}$. Let $a_1\cdots,a_n\in\mathcal{A}$ such that $a_j\in\mathcal{A}_{i_j}$ with $i_1,\dots ,i_n\in I$ with $i_{j-1} < i_j > i_{j+1}$.  For each $j$, we let $A_j=\begin{bmatrix}
a_j & 0 \\
0 & a_j
\end{bmatrix}$. By our assumption, we have
\begin{equation}\label{meqn2}
\widetilde{E}(A_1\cdots A_n)=\widetilde{E}(A_1\cdots A_{j-1}\widetilde{E}(A_j)A_{j+1}\cdots A_n).
\end{equation}
The left hand of equation \eqref{meqn2} is 
$$
\widetilde{E}\Bigg(
\begin{bmatrix}
a_1\cdots a_n & 0 \\
0 & a_1\cdots a_n
\end{bmatrix}
\Bigg)  = 
\begin{bmatrix}
E(a_1\cdots a_n) & E'(a_1\cdots a_n) \\
0 & E(a_1\cdots a_n)
\end{bmatrix}.
$$ On the other hand, the right hand side of equation \eqref{meqn2} is 
$$ \begin{bmatrix}
E(a_1\cdots a_{j-1}E(a_j)a_{j+1}\cdots a_n) & E'(a_1\cdots a_{j-1}E(a_j)a_{j+1}\cdots a_n )+E(a_1\cdots a_{j-1}E'(a_j)a_{j+1}\cdots a_n ) \\
0 & E(a_1\cdots a_{j-1}E(a_j)a_{j+1}\cdots a_n )
\end{bmatrix},$$
where we can use \eqref{long.formula} specialized to $a'_1=\dots a'_n=0$ to easily compute the $(1,2)$-entry. Therefore, by checking equality \eqref{meqn2} in each entry we obtain \eqref{eqn.ov.monotone} and \eqref{eqn.ov.inf.monotone}.
\end{proof}

We are now able to prove Theorem \ref{Thm.monotone.add.conv}  describing the OVI monotone convolution using the infinitesimal Cauchy transform.

\begin{proof}[Proof of Theorem \ref{Thm.monotone.add.conv}]
Suppose that $(\mathcal{A},\mathcal{B},E,E')$ is a $C^*$-operator-valued infinitesimal probability space, and $x,y\in\mathcal{A}$ are selfadjoint random varaibles that are infinitesimally monotone independent over $\mathcal{B}$. Let $X=
\begin{bmatrix}
x & 0 \\
0 & x
\end{bmatrix}$ and $Y=
\begin{bmatrix}
y & 0 \\
0 & y
\end{bmatrix}$. By Proposition \ref{p2}, $X$ and $Y$ are monotone independent over $\widetilde{\mathcal{B}}$.
By \eqref{Meqn}, we have
\begin{equation} \label{eqn3}
F_{X+Y}\Bigg (
\begin{bmatrix}
b & c \\
0 & b
\end{bmatrix}\Bigg )  \\
= 
F_X\Bigg (F_Y\Bigg (
\begin{bmatrix}
b & c \\
0 & b
\end{bmatrix}\Bigg ) \Bigg).
\end{equation}
Note that 
\begin{eqnarray*}
& & F_X\Bigg (F_Y\Bigg (
\begin{bmatrix}
b & c \\
0 & b
\end{bmatrix}\Bigg ) \Bigg)  
=
F_X\Bigg (G_Y\Bigg (
\begin{bmatrix}
b & c \\
0 & b
\end{bmatrix}\Bigg )^{-1} \Bigg)  
=
F_X\Bigg ( \begin{bmatrix}
G_y(b) & G'_y(b)c+g_y(b) \\
0 & G_y(b)
\end{bmatrix}^{-1} \Bigg )  \\
&=& F_X \Bigg ( \begin{bmatrix}
F_y(b) & -F_y(b) (G_y'(b)c+g_y(b))F_y(b) \\
0 & F_y(b)
\end{bmatrix} \Bigg ) \\
&=&
\begin{bmatrix}
F_x(F_y(b)) & -F_x(F_y(b))\Big (G_x'(F_y(b))(-F_y(b)(G_y'(b)c+g_y(b))F_y(b)))+g_x(F_y(b)) \Big )F_x(F_y(b)) \\
0 & F_x(F_y(b))
\end{bmatrix}.
\end{eqnarray*}
Comparing to the $(1,2)$-entry of \eqref{eqn3}, we obtain
\begin{equation*}
-F_{x+y}(b)g_{x+y}(b)F_{x+y}(b) = -F_x(F_y(b))\Big(G_x'(F_y(b))(-F_y(b)g_y(b)F_y(b))+g_x(F_y(b)) \Big)F_x(F_y(b)).
\end{equation*}
Thus, 
\begin{equation}
g_{x+y}(b) = G_x'(F_y(b))(-F_y(b)g_y(b)F_y(b))+g_x(F_y(b)).
\end{equation}
\end{proof}

\begin{rem}
When $\mathcal{B}=\mathbb{C}$, we obtain the scalar version of infinitesimal Boolean convolution. Precisely, if $x=x^*$ and $y=y^*$ are infinitesimally monotone from a $C^*$-infinitesimal probability space $(\mathcal{A},\varphi,\varphi')$, we have
$$
g_{x+y}(z)=-g_y(z)F_y(z)^2 G_x'(F_y(z))+g_x(F_y(z)).
$$
\end{rem}

\subsection{OVI Monotone Cumulants}

$\ $

\noindent OVI monotone cumulants are defined using the OV monotone cumulants as follows

\begin{defn}
Suppose that $(\mathcal{A},\mathcal{B},E,E')$ is an OVI probability space, with operator-valued monotone cumulants $\{h_n^{\mathcal{B}}:\mathcal{A}^n\to \mathcal{B}\}$ (see equation \eqref{mom-mon}). We define the \emph{operator-valued infinitesimally monotone cumulants} $\{\partial h_n^{\mathcal{B}}:\mathcal{A}^n\to \mathcal{B}\}$ to be the (unique) family of functions such that for all $n\in\mathbb{N}$ and $a_1,\dots, a_n\in\mathcal{A}$ 
\begin{equation}
\label{inf.mom-mon}
E'(a_1\cdots a_n) =\sum_{\pi\in\mathcal{NC}(n)} \frac{1}{\tau(\pi)!} \partial h_\pi^{\mathcal{B}}(a_1,\dots,a_n).
\end{equation}
\end{defn}

This equation is actually defining (uniquely) the infinitesimal cumulants, since in the right hand side we can isolate the term $h'_n(a_1,\dots,a_n)$. Thus we can obtain $h'_n$ recursively from $E'$, the cumulants $h_1,\dots,h_{n-1},$ and the previously defined $h'_1,\dots,h'_{n-1}$. 

\begin{rem}
Similar to the Boolean and free case, if we consider the operator-valued monotone cumulants $\{h^{\widetilde{\mathcal{B}}}_n:\widetilde{\mathcal{A}}^n\to \widetilde{\mathcal{B}}\}$ of the corresponding upper triangular operator-valued probability space $(\widetilde{\mathcal{A}},\widetilde{\mathcal{B}},\widetilde{E})$, we get that $h^{\widetilde{\mathcal{B}}}_n=\widetilde{h}_n$ where $\widetilde{h}_n$ are the multilinear functional obtained from merging the OV and OVI cumulants of $(\mathcal{A},\mathcal{B},E,E')$ as in Notation \ref{nota.tilde.cumulant}. From now on we will use $h^{\tilde{\mathcal{B}}}_n$ or $\widetilde{h}_n$ indistinctly to refer to monotone cumulants of the associate upper triangular probability space.
\end{rem}


In analogy to the relation holding between Scalar-Valued and OV infinitesimal independences (Theorem 13 of \cite{tseng2019operator} in the free case and Theorem \ref{thm.relation.inf.boo.matrices} in the Boolean case), we might wonder if the same is true for the monotone case:

\begin{conj}
\label{conj.rel.scalar.OVI.monotone}
Suppose that $(\mathcal{M},\varphi,\varphi')$ is an infinitesimal probability space.  
If the sets $\{a^{(1)}_{i,j}\colon1\le i,j\le N\},\dots, \{a^{(k)}_{i,j}\colon1\le i,j\le N\}$ are infinitesimally monotone in $\mathcal{M}$, then the elements $[a^{(1)}_{i,j}]_{i,j=1}^N,$ $\dots,[a^{(k)}_{i,j}]_{i,j=1}^N$ are infinitesimally monotone in $M_n(\mathcal{M}).$
\end{conj}

However, as opposed to Boolean and free cases, monotone independence cannot be characterized in terms of vanishing of monotone cumulants, and this was a key step to prove the analogue results. Therefore the approach using cumulants is not useful here. The best we can say in this direction is that the following proposition still holds in the monotone setting.

\begin{prop}
Suppose $A^{(k)}=[a_{i,j}^{(k)}]$ are elements in $\mathcal{A}=M_N(\mathcal{M})$ for all $k=1,\dots,n$. Then the following formula holds:
\begin{equation} \label{beqn2}
[\partial h^{\mathcal{B}}_n(A^{(1)},A^{(2)},\ldots ,A^{(n)})]_{i,j}= \sum\limits_{i_2,i_3,\ldots i_n=1}^{N}h'_{n}(a_{i,i_2}^{(1)},a_{i_2,i_3}^{(2)}\ldots ,a_{i_n,j}^{(n)}).
\end{equation}
\end{prop}

The proof follows the same lines of the free and Boolean case. Although, cumulants approach is not very useful we suspect that Conjecture \ref{conj.rel.scalar.OVI.monotone} can be proved using the equations defining OVI monotone independence in a clever way.

\section{Infinitesimal Central Limit Theorem}
\label{sec:InfCLT}
This section has two parts, first we discuss the operator-valued infinitesimal version of the Central Limit Theorem \ref{OVCLT} (OVI CLT), and then we specialize to the scalar-valued case. 

\subsection{Operator-Valued Infinitesimal CLT}

$\ $

\noindent As usual, the main idea is to consider the upper triangular space $(\widetilde{\mathcal{A}},\widetilde{\mathcal{B}},\widetilde{E})$ and then apply the tools we already know from the OV setting. The extra challenge now is that we require $\widetilde{\mathcal{A}}$ to be a Banach algebra (see Remark \ref{rem.CLT.Banach}). Therefore, given a Banach algebra $\mathcal{A}$, we need to equip $\widetilde{\mathcal{A}}$ with a suitable topology. This can be easily handled by defining the norm $\Vert \cdot \Vert_{\widetilde{\mathcal{A}}}$ on $\widetilde{\mathcal{A}}$ as 
$$ \left\Vert \begin{bmatrix}a & a' \\ 0 & a \end{bmatrix} \right\Vert_{\widetilde{\mathcal{A}}}:= \Vert a\Vert_{\mathcal{A}} + \Vert a' \Vert_{\mathcal{A}} \qquad \forall  a,a'\in \mathcal{A}.$$
We will omit the subindex in the norm, as it is clear that we always use $\Vert \cdot \Vert_{\widetilde{\mathcal{A}}}$ on upper triangular matrices and $\Vert \cdot \Vert_{\mathcal{A}}$ on elements of $\mathcal{A}$.

\begin{rem}
Let us show that $( \widetilde{\mathcal{A}},\Vei \cdot \Ved )$ is indeed a Banach algebra. It is clear that $\Vert \cdot \Vert_{\widetilde{\mathcal{A}}}$ is a norm on $\widetilde{\mathcal{A}}$. moreover, note that for every two elements $A=\begin{bmatrix}a & a' \\ 0 & a \end{bmatrix}$ and $B=\begin{bmatrix}b & b' \\ 0 & b \end{bmatrix}$ in $\widetilde{\mathcal{A}}$, the required inequality for the norm holds:
$$\Vei AB\Ved = \Vei \begin{bmatrix}ab & ab'+ a'b \\ 0 & ab \end{bmatrix}\Ved= \Vei ab \Ved + \Vei ab'+a'b\Ved $$
$$  \leq  \Vei a\Ved \Vei b \Ved + \Vei a\Ved \Vei b'\Ved+ \Vei a'\Ved \Vei b \Ved + \Vei a'\Ved \Vei b'\Ved =\Vei A\Ved\Vei B\Ved.$$
In addition, for every Cauchy sequence $\{A_n\}_{n\geq 1}\subset \widetilde{\mathcal{A}}$ with $A_n=\begin{bmatrix}a_n & a'_n \\ 0 & a_n \end{bmatrix}$, we have that
$$ \Vei A_n-A_m \Ved=\Vei \begin{bmatrix}a_n -a_m& a'_n-a'_m \\ 0 & a_n-a_m \end{bmatrix} \Ved = \Vei a_n-a_m \Ved+\Vei a'_n-a'_m \Ved.$$
Then we get that $\{a_n\}_{n\geq 1}$ and $\{a'_n\}_{n\geq 1}$ are Cauchy and thus convergent (since $\mathcal{A}$ is complete). This implies that $\{A_n\}_{n\geq 1}\subset \widetilde{\mathcal{A}}$ is convergent, and means that $( \widetilde{\mathcal{A}},\Vei \cdot \Ved )$ is complete. Therefore, we conclude that $( \widetilde{\mathcal{A}},\Vei \cdot \Ved )$ is a Banach algebra. 
\end{rem}

For the infinitesimal case we also need to record the infinitesimal variance. Let $(\mathcal{A},\mathcal{B},E,E')$ be an OVI probability space (where $\mathcal{A}$ and $\mathcal{B}$ are Banach algebras), and let $a\in\mathcal{A}$ be a centered and infinitesimally centered variable, namely $E(a)=E'(a)=0$. We record its variance and infinitesimal variance as two maps
$$ \eta,\eta':\mathcal{B} \to \mathcal{B}, \qquad \eta(b)=E(a_i ba_i), \qquad \eta'(b)=E'(a_i ba_i).$$
Then in complete analogy to Definition \ref{defn.OV.central.laws} we can describe the central limit laws. 
\begin{defn} Let $(\mathcal{A},\mathcal{B},E,E')$ be an OVI probability space, let $\eta,\eta':\mathcal{B} \to \mathcal{B}$ be the variance and infinitesimal variance of a variable $a$ as described above, and let $\nu'_{\eta,\eta'}:\mathcal{B}\langle X\rangle \to \mathcal{B}$ be a symmetric distribution, namely 
\[
\nu'_{\eta,\eta'}(b_0Xb_1Xb_2\cdots b_{k-1}Xb_k) = 0 \qquad  \text{ for all odd } k\in\mathbb{N},\ b_0,\dots,b_k\in \mathcal{B}.
\]
Then we say that
\begin{itemize}
    \item [(1)] $\nu'_{\eta,\eta'}$ is the centered \emph{$\mathcal{B}$-valued infinitesimal semi-circle law} with infinitesimal variance $(\eta,\eta')$ if
    \begin{equation*}
        \nu'_{\eta,\eta'}(b_0Xb_1Xb_2\cdots b_{k-1}Xb_k) = \sum\limits_{\pi\in\mathcal{NC}_2(k)}b_0\partial\eta_{\pi}(b_1,\cdots,b_{k-1})b_k,  \qquad  \text{ for all even } k\in\mathbb{N}.
    \end{equation*}
    \item [(2)] $\nu'_{\eta,\eta'}$ is the centered \emph{$\mathcal{B}$-valued infinitesimal Bernoulli law} with infinitesimal variance $(\eta,\eta')$ if
    \begin{equation*}
         \nu'_{\eta,\eta'}(b_0Xb_1Xb_2\cdots b_{k-1}Xb_k) =  \sum\limits_{j=0}^{k/2-1}b_0\eta(b_1)\cdots \eta'(b_{2j+1})\cdots \eta(b_{k-1})b_k,  \qquad  \text{ for all even } k\in\mathbb{N}.
    \end{equation*}
    \item [(3)] $\nu'_{\eta,\eta'}$ is the centered \emph{$\mathcal{B}$-valued infinitesimal arcsine law} with infinitesimal variance $(\eta,\eta')$ if
    \begin{equation*}
         \nu'_{\eta,\eta'}(b_0Xb_1Xb_2\cdots b_{k-1}Xb_k) = \sum\limits_{\pi\in\mathcal{NC}_2(k)}\frac{1}{\tau(\pi)!}b_0\partial\eta_{\pi}(b_1,\cdots,b_{k-1})b_k,  \qquad  \text{ for all even } k\in\mathbb{N}.
    \end{equation*}
\end{itemize}
\end{defn}

As one should expect, the previously defined laws arise as the central limit  in the OVI setting. 
\begin{thm}\label{OVICLT}
Suppose that $(\mathcal{A},\mathcal{B},E,E')$ is an OVI probability space (where $\AA$ and $\mathcal{B}$ are Banach algebras), and $\{a_i\}_{i=1}^\infty$ is a sequence of centered ($E(a_i)=0$), infinitesimally centered ($E'(a_i)=0$), identically distributed infinitesimally freely independent (resp, Boolean, monotone) elements. Consider
$$ \eta,\eta':\mathcal{B} \to \mathcal{B}, \qquad \eta(b)=E(a_i ba_i), \qquad \eta'(b)=E'(a_i ba_i), $$
the common variance and infinitesimal variance of the $a_i$. Then
for all $b_0,\cdots,b_k\in\mathcal{B}$ we have
\begin{eqnarray}
&E(b_0S_Nb_1S_Nb_2\cdots b_{k-1}S_Nb_k)&\longrightarrow \nu_{\eta}(b_0Xb_1Xb_2\cdots b_{k-1}Xb_k ) \text{ as } N\to \infty, \\
&E'(b_0S_Nb_1S_Nb_2\cdots b_{k-1}S_Nb_k)&\longrightarrow \nu'_{\eta,\eta'}(b_0Xb_1Xb_2\cdots b_{k-1}Xb_k ) \text{ as } N\to \infty,
\end{eqnarray}
where $\nu_{\eta}$ is $\mathcal{B}$-valued semi-circle law (resp, Bernoulli law, arcsine law) and  
$\nu_{\eta,\eta'}$ is $\mathcal{B}$-valued infinitesimal semi-circle law (resp, Bernoulli law, arcsine law), with variances $\eta,\eta'$.
\end{thm}
\begin{proof}
We will just prove the free case, as the procedure for Boolean and monotone is very similar. Consider the upper triangular probability space $(\widetilde{\mathcal{A}},\widetilde{\mathcal{B}},\widetilde{E})$ and for each $j$ we let 
$A_j=\begin{bmatrix}
a_j & 0 \\
0 & a_j
\end{bmatrix}$. Note that $\{A_j\}_{j=1}^{\infty}$ is freely independent with    
$\widetilde{E}(A_j)=0$. 
In addition, 
$$
\widetilde{S}_N:=\frac{1}{\sqrt{N}}(A_1+\cdots+A_N)=
\begin{bmatrix}
S_N & 0 \\
0 & S_N 
\end{bmatrix}.$$
Moreover, for a fixed $k\in\mathbb{N}$ if we let $b_0,b_1,\cdots, b_k\in\mathcal{B}$ and denote $B_j:=\begin{bmatrix} 
b_j & 0 \\ 0 & b_j \end{bmatrix}$ for each $j$, then
\begin{eqnarray*}
\widetilde{E}(B_0\widetilde{S}_NB_1\cdots \widetilde{S}_NB_k)
&=& \widetilde{E}\Big( \begin{bmatrix}
b_0S_Nb_1\cdots S_Nb_k & 0 \\
0 & b_0S_Nb_1\cdots S_Nb_k
\end{bmatrix} \Big)  \\
&=& 
\begin{bmatrix}
E(b_0S_Nb_1\cdots S_Nb_k) & E'(b_0S_Nb_1\cdots S_Nb_k)  \\
0 & E(b_0S_Nb_1\cdots S_Nb_k)
\end{bmatrix}.
\end{eqnarray*}

By Theorem \ref{OVCLT}, $\widetilde{E}(B_0\widetilde{S}_NB_1\cdots \widetilde{S}_NB_k) \longrightarrow \widetilde{\nu}_{\eta}(B_0XB_1X\cdots XB_k)$ as $N\to \infty$ where
\[
\widetilde{\nu}_{\eta}(B_0XB_1X\cdots XB_k) = \begin{cases} 0 & \text{if }k \text{ is odd,} \\ 
                                                        \sum\limits_{\pi\in\mathcal{NC}_2(k)}B_0\widetilde{\eta}_{\pi}(B_1,\cdots,B_{k-1})B_k  & \text{if }k \text{ is even.}\end{cases}
\]
Note that if $k$ is even, then
\begin{eqnarray*}
\sum\limits_{\pi\in\mathcal{NC}_2(k)}B_0\widetilde{\eta}_{\pi}(B_1,\cdots,B_{k-1})B_k &=& \sum\limits_{\pi\in\mathcal{NC}_2(k)}B_0\widetilde{E}_{\pi}(A_1B_1,A_1B_2,\dots,A_1B_{k-1},A_1)B_k \\
&=& \sum\limits_{\pi\in\mathcal{NC}_2(k)}
\begin{bmatrix}
E_{\pi} (a_1b_1,\dots,a_1b_{k-1},a_1) & \partial E_{\pi} (a_1b_1,\dots,a_1b_{k-1},a_1) \\
0 & E_{\pi} (a_1b_1,\dots,a_1b_{k-1},a_1)
\end{bmatrix} \\
&=& \begin{bmatrix}
\sum\limits_{\pi\in\mathcal{NC}_2(k)} b_0\eta_{\pi}(b_1,\dots,b_{k-1})b_k & \sum\limits_{\pi\in\mathcal{NC}_2(k)}  b_0\partial\eta_{\pi}(b_1,\dots,b_{k-1})b_k \\
0 & \sum\limits_{\pi\in\mathcal{NC}_2(k)} b_0\eta_{\pi}(b_1,\dots,b_{k-1})b_k
\end{bmatrix}.
\end{eqnarray*}
Thus, we conclude that 
\begin{eqnarray*}
&E(b_0S_Nb_1S_Nb_2\cdots b_{k-1}S_Nb_k)&\longrightarrow \nu_{\eta}(b_0Xb_1Xb_2\cdots b_{k-1}Xb_k ) \text{ as } N\to \infty \ ; \\
&E'(b_0S_Nb_1S_Nb_2\cdots b_{k-1}S_Nb_k)&\longrightarrow \nu'_{\eta,\eta'}(b_0Xb_1Xb_2\cdots b_{k-1}Xb_k ) \text{ as } N\to \infty \ .
\end{eqnarray*}

\end{proof}

\subsection{Scalar-Valued Infinitesimal CLT}
\label{ssec:SVICLT}

$\ $

\noindent Note that we can view the scalar-valued infinitesimal central limit theorem as a special case of operator-valued infinitesimal central limit theorem. In this subsection, we provide another approach to construct the scalar-valued infinitesimal central limit theorem.    

\begin{prop}\label{prop.scalar.CLT}
Suppose that $(\mathcal{A},\varphi,\varphi')$ is an infinitesimal non-commutative probability space, and $a,a_1,a_2,\dots $ is a sequence of identically distributed elements in $\mathcal{A}$ which are infinitesimally independent. We also assume that $\varphi(a)=\varphi'(a)=0, \varphi(a^2)=1$, and $\varphi'(a^2)=\alpha$. Let $G$ be the Cauchy transform and $g$ be the infinitesimal Cauchy transform of the limit law of
$S_N=(a_1+\cdots +a_N)/\sqrt{N}$ with respect to $(\varphi,\varphi')$. Then 
$$
g(z) = \frac{-\alpha}{2}\frac{d}{dz}\Big( zG(z)\Big).
$$
\end{prop}

\begin{proof}
Let $(\kappa_n)_{n\geq 1}$ represent the free $(r_n)_{n\geq 1}$, Boolean $(\beta_n)_{n\geq 1}$ or monotone $(h_n)_{n\geq 1}$ cumulants. Then 
\begin{eqnarray*}
\kappa_n(S_N,\cdots,S_N) &=&  (\frac{1}{\sqrt{N}})^n \kappa_n(a_1+\cdots+a_N,\cdots, a_1+\cdots+a_N) \\
                    &=&  N^{1-n/2} \kappa_n(a_1,\cdots,a_1).
\end{eqnarray*}
Then, for $n\geq 2$ we have that $\kappa_n(S_N,\cdots,S_N)\longrightarrow 0$ as $N\rightarrow \infty$. On the other hand, \[ \kappa_1(a)=\varphi(a)=0 \text{ and } \kappa_2(a,a)=\varphi(a)-\varphi(a)^2=1. \]  
Similarly, $\kappa'_1(a)=0$ and $\kappa'_2(a,a)=\alpha$. 
To compute the moments, we separate in cases: 
\begin{itemize}
    \item For the free case, one has
$$ 
m'_n = \sum\limits_{\pi\in NC(n)} \partial r_{\pi} = \sum\limits_{\pi \in NC_2(n)} \frac{n}{2} r_2^{n/2-1}r'_2 = \frac{n}{2}\alpha |NC_2(n)|=\frac{n}{2}\alpha m_n. 
$$
    \item For the Boolean case, we have
$$ 
m'_n = \sum\limits_{\pi\in I(n)} \partial \beta_{\pi} = \sum\limits_{\pi \in I_2(n)} \frac{n}{2} \beta_2^{n/2-1}\beta'_2 = \frac{n}{2}\alpha |I_2(n)|=\frac{n}{2}\alpha m_n. 
$$
    \item For the monotone case, we obtain
$$ 
m'_n = \sum\limits_{\pi \in NC(n)} \frac{1}{\tau(\pi)!}\partial h_{\pi} = \sum\limits_{\pi \in NC_2(n)}  \frac{1}{\tau(\pi)!}\frac{n}{2} h_2^{n/2-1}h'_2 = \frac{n}{2}\alpha  \sum\limits_{\pi \in NC_2(n)}  \frac{1}{\tau(\pi)!} =\frac{n}{2}\alpha m_n. 
$$
\end{itemize}
Hence, note that for each case of independence, we have for $n\in \mathbb{N}$,  
$$ m'_n=\frac{n}{2}\alpha m_n.$$ 
Notice that the values are zero if $n$ is odd. We conclude that 
\begin{eqnarray}
g(z) &=& \sum\limits_{n=1}^\infty \frac{m'_n}{z^{n+1}} = \sum\limits_{n=1}^\infty \frac{\alpha n}{2}\frac{m_n}{z^{n+1}} = \frac{-\alpha}{2}\sum\limits_{n=1}^\infty {(-n)}\frac{m_n}{z^{n+1}} = \frac{-\alpha}{2}\frac{d}{dz}\Big( zG(z)\Big). 
\end{eqnarray}
\end{proof}


Let $\mathcal{E}(\mathbb{R})$ be the space of all complex-valued smooth functions on $\mathbb{R}$, and then the space of all continuous linear functionals on $\mathcal{E}(\mathbb{R})$ is denoted by $\mathcal{E}'(\mathbb{R})$. It is known that the elements of $\mathcal{E}'(\mathbb{R})$ are distributions with compact support, for a detailed discussion on this we refer to \cite{CM89}.   
Note that if we consider $f:[a,b]\to \mathbb{R}$ is a function of bounded variation, and then we let $f=0$ on the complement of $[a,b]$. 
For such $f$, we define $\nu:\mathcal{E}(\mathbb{R})\to \mathbb{C}$ by
$$
\nu(\phi)=\int_{\mathbb{R}} \phi(t) f''(t)dt
$$
where the second derivative of $f$ is taken in the distributional sense. It is easy to see that $\nu$ is a distribution and $supp(\nu)\subset [a,b]$, so $\nu\in\mathcal{E}'(\mathbb{R})$. We denote it by $\nu=f''$, this set of distributions have been studied in \cite{BS} under the name $\mathcal{M}_2$.

Given $\nu=f''$ for some function $f$ of bounded variation on a closed interval, we define the infinitesimal Cauchy transform $g_{\nu}$ of the distribution $\nu$ by 
$$g_{\nu}(z) =\nu\left(\frac{1}{z-t}\right), \qquad \forall Im(z)\neq 0.$$
Then, we also see that 
\begin{equation*}
g_{\nu}(z) = \int \frac{1}{z-t}f''(t)dt = -\int \frac{d}{dt}\Big(\frac{1}{z-t}\Big)f'(t)dt  \\
           = -\int \frac{-1}{(z-t)^2}f'(t)dt = \frac{d}{dz} \Bigg( \int \frac{1}{z-t}f'(t)dt\Bigg).
\end{equation*}
Now, let us apply the Proposition \ref{prop.scalar.CLT} to deduce the infinitesimal central limit theorem. The free case was already proved by Popa \cite{POP10}, another proof can be found in \cite{BS}.

\begin{thm}\label{SVICLT}
Suppose that $(\mathcal{A},\varphi,\varphi')$ is an infinitesimally non-commutative probability space. Let $a_1,a_2,\dots $ be a sequence in $\mathcal{A}$ with identical distribution, which are infinitesimally independent (free, Boolean, monotone). We also assume that $\varphi(a_1)=\varphi'(a_1)=0, \varphi(a_1^2)=1$, and $\varphi'(a_1^2)=\alpha.$ For each $N\in \mathbb{N}$, we let $S_N=(a_1+\cdots +a_N)/\sqrt{N}$, then for every $k\in\mathbb{N}$, we have
\begin{equation*}
\lim\limits_{N\rightarrow \infty}\varphi'(S_N^k) = \int t^k d\mu'(t),
\end{equation*}
where $\mu'$ is stated as follows.  
\begin{itemize}
    \item Free case: $\mu'$ is a signed measure such that $\mu'=\alpha (\mu_1-\mu_2)$ with 
$$
d\mu_1(t)=\frac{1}{2\pi}\frac{t^2}{\sqrt{4-t^2}}1_{(-2,2)}(t)dt \ \text{and} \ \ d\mu_2(t)=\frac{1}{\pi\sqrt{4-t^2}}1_{(-2,2)}(t)dt.
$$
    \item Boolean case: $\mu'=f''$ with 
$$df(t) = \frac{\alpha}{2}(\delta_{-1}-\delta_{1})(t).$$
    \item Monotone case: $\mu'=f''$ with 
$$df(t)=\frac{-\alpha}{\pi}\frac{t}{2\sqrt{2-t^2}}1_{(-\sqrt{2},\sqrt{2})}(t)dt.$$ 
\end{itemize}

\end{thm}

\begin{proof}

{\bf Free case:} 
Recall that the law of free central limit theorem is the standard semi-circle law, and then the corresponding Cauchy transform is
$$
G(z)=\int_{-2}^2 \frac{1}{z-t}\frac{1}{2\pi}\sqrt{4-t^2}dt = \frac{z-\sqrt{z^2-4}}{2}.
$$
Hence by Proposition \ref{prop.scalar.CLT}, we have
$$
g(z)= \frac{-\alpha}{2}\frac{d}{dz}\Big( zG(z)\Big) = \frac{\alpha}{2} \Big(\frac{z^2-2}{\sqrt{z^2-4}}-z\Big).
$$
Note that if we consider the following two probability measures 
$$
d\mu_1(t)=\frac{1}{2\pi}\frac{t^2}{\sqrt{4-t^2}}1_{(-2,2)}(t)dx \quad \text{and} \quad d\mu_2(t)=\frac{1}{\pi\sqrt{4-t^2}}1_{(-2,2)}(t)dt,
$$
Then  
\begin{eqnarray*}
\int_{-2}^2 \frac{1}{z-t}d(\mu_1-\mu_2)(t) &=& \int_{-2}^2 \frac{1}{(z-t)} \Big(\frac{t^2}{2}-1 \Big)\frac{1}{\pi \sqrt{4-t^2}}dt \\
                                           &=& \int_{-2}^2 \frac{t^2}{2} \frac{1}{(z-t)} \frac{1}{\pi \sqrt{4-t^2}}dt - \int_{-2}^2 \frac{1}{(z-t)} \frac{1}{\pi \sqrt{4-t^2}}dt \\
                                           &=& \Big( \frac{z^2}{\sqrt{z^2-4}}-z \Big)-\Big( \frac{1}{\sqrt{z^2-4}}  \Big) \\
                                           &=& \frac{1}{2} \Big(\frac{z^2-2}{\sqrt{z^2-4}}-z\Big).
\end{eqnarray*}
Thus, if we let $\mu'=\alpha ( \mu_1-\mu_2)$, then $\mu'$ is a signed measure such that 
$$
g(z) =\int \frac{1}{z-t}d\mu'(t).
$$

{\bf Boolean case:} 

The law of Boolean central limit theorem is $\frac{1}{2}\delta_1+\frac{1}{2}\delta_{-1}$, which implies that
$$
G(z)= \int \frac{1}{z-t}d(\frac{1}{2}\delta_1+\frac{1}{2}\delta_{-1})(t) = \frac{1}{z-1}+\frac{1}{z+1}.
$$
Hence, we apply the Proposition \ref{prop.scalar.CLT} that we get
$$
g(z) = \frac{-\alpha}{2}\frac{d}{dz}\Big(zG(z) \Big) = \frac{2\alpha z}{(z^2-1)^2}.
$$
Note that 
\[
\int \frac{1}{z-t}d(\frac{\alpha}{2}\delta_{-1}-\frac{\alpha}{2}\delta_{1})(t)
=\frac{\alpha}{2} \Big( \frac{1}{z-(-1)}-\frac{1}{z-1} \Big) 
= \frac{\alpha}{2} \Big( \frac{1}{1-z}+\frac{1}{1+z} \Big) 
=\frac{2\alpha}{2(1-z^2)}
=\frac{2\alpha}{2-2z^2}, 
\]
and also
$$
\frac{d}{dz}\Big( \frac{2\alpha}{2-2z^2} \Big) = \frac{2\alpha z}{(z^2-1)^2}.
$$
Thus, we obtain
$$
\frac{d}{dz}\Bigg( \int \frac{1}{z-t}d(\frac{\alpha}{2}\delta_{-1}-\frac{\alpha}{2}\delta_{1})(t) \Bigg) = g(z).
$$


{\bf Monotone case:} 

It is known that the Cauchy transform of the monotone central limit theorem is 
$$
G(z) = \int_{-\sqrt{2}}^{\sqrt{2}} \frac{1}{z-t}\frac{1}{\pi \sqrt{2-t^2}}dt = \frac{1}{\sqrt{z^2-2}}.
$$
Applying the Proposition \ref{prop.scalar.CLT}, we obtain
$$
g(z) = \frac{-\alpha}{2}\frac{d}{dz}\Big(zG(z) \Big) = \frac{\alpha}{(z^2-2)^{3/2}}.
$$
Observe that 
\[
\int_{-\sqrt{2}}^{\sqrt{2}} \frac{1}{z-t}\frac{-\alpha}{\pi}\frac{t}{2\sqrt{2-t^2}}dt  =\frac{-\alpha}{\pi}\Bigg(\frac{\pi}{2}\frac{z}{\sqrt{z^2-2}}-\frac{\pi}{2} \Bigg)
= \frac{-\alpha}{2}\frac{z}{\sqrt{z^2-2}}+\frac{\alpha}{2},
\]
and also
$$
\frac{d}{dz}\Big( \frac{-\alpha z}{2\sqrt{z^2-2}}+\frac{\alpha}{2}\Big) = \frac{\alpha}{(z^2-2)^{3/2}}.
$$
Hence, we have
$$
\frac{d}{dz}\Bigg( \int_{-\sqrt{2}}^{\sqrt{2}} \frac{1}{z-t}\frac{-\alpha}{\pi}\frac{t}{2\sqrt{2-t^2}}dt \Bigg) = g(z).
$$

\end{proof}

\section{Relation among OVI cumulants}
\label{sec:Relation.cumulants}
In this section we discuss relations between non-commutative cumulants in the OVI setting. These relations follow directly from the use of the upper triangular OV space, and the relations at the OV level. 

\begin{defn}
The \textit{min-max order} $\ll$ in the lattice $NC(n)$ is defined as follows. For $\pi,\sigma\in \mathcal{NC}(n)$, we write `$\pi \ll \sigma$' to mean that $\pi \leq \sigma$ and that for every block $V$ of $\sigma$ there exists a block $W$ of $\pi$ such that $\min(V),\max(V) \in W$. 

We will say that a partition $\pi$ is \textit{irreducible} if $\pi\ll 1_n$. This is equivalent to the fact that $1$ and $n$ are in the same block of $\pi$. The set of non-crossing irreducible partitions is denoted by $\mathcal{NC}_{irr}(n)$.
\end{defn}

We begin by recalling that the formulas relating Boolean and free cumulants on the OV setting have already appeared in connection to the study of the Boolean Bercovici-Pata bijection \cite{ABFN}. 

\begin{prop}[Proposition 5.4 of \cite{ABFN}]
\label{Prop.relation.boofree.cumulants}
Suppose that $(\mathcal{A},\mathcal{B},E,)$ is an OV probability space, and let $\{ r^{\mathcal{B}}_n\}_{n \ge 1}$ and $\{ \beta^{\mathcal{B}}_n\}_{n \ge 1}$, be the families of OVI free and Boolean cumulants, respectively. Then, the following relations among them hold:
\begin{eqnarray}
	\beta^{\mathcal{B}}_n (a_1, \dots, a_n) 
		&=& \sum_{\pi\in \mathcal{NC}_{irr}(n)} r^{\mathcal{B}}_\pi(a_1, \dots, a_n), \label{OVBooFree}\\
		r^{\mathcal{B}}_n (a_1, \dots, a_n) 
		&=& \sum_{\pi\in \mathcal{NC}_{irr}(n)} (-1)^{|\pi|-1} \beta^{\mathcal{B}}_\pi(a_1, \dots, a_n), \label{OVFreeBoo}
\end{eqnarray}
for every $n$ and elements $a_1,\dots, a_n\in \mathcal{A}$.
\end{prop}

We now proceed to obtain the remaining formulas in the OV level, which is the main purpose of this section, and then transfer them to the OVI level. Many of the usual ideas in the scalar-valued case carryover to the OV setting, the main difficulty is to deal with the technical issues that arose from the non-commutativity of the operators (opposed to the commutativity of the scalar values). The following technical Lemma is very useful when dealing with these issues. \\

Let us use the notation
$NC^{(2)}(n):=\{(\sigma,\pi)\in NC(n)^2: \sigma \leq \pi\}$
and $NC^{(2)}:=\bigsqcup_{n=1}^\infty NC^{(2)}(n)$.

\begin{lem}
\label{Lemma.useful}
Let $\{f_\pi\}_{\pi\in NC}$ be a multiplicative family of $\mathcal{B}$-bimodule maps, and let $c:NC\to\mathbb{C}$ be an arbitrary map.  Define $C:NC^{(2)}\to \mathbb{C}$ by $C(\sigma,\pi)=\prod_{V\in\pi} c(\sigma|V)$ and consider the family 
$$g_\pi:= \sum_{\substack{\sigma\in NC(n)\\ \sigma\leq \pi}} C(\sigma,\pi) f_\sigma, \qquad \forall \pi \in NC.$$
Then the family $\{g_\pi\}_{\pi\in NC}$ is also multiplicative. 
\end{lem}

\begin{rem}
The previous Lemma is a slight generalization of Proposition 2.1.7 in \cite{Spe98}, which is restated here in Remark \ref{rem.spei.mob}. The first part of this remark is the special case where $c(\pi)=1$ $\forall\pi\in NC$.
\end{rem}

\begin{proof}
We follow the same ideas of (Proposition 2.1.7, \cite{Spe98}). To corroborate that $g_\pi$ is multiplicative, we only need to check that the recurrence holds. So take $\pi\in NC(n)$, pick $V=\{l+1,\dots l+k\}\in \pi$ a interval block of $\pi$, and set $\pi'=\pi\backslash V\in NC(n-k)$. y substituting, $g_{\pi'},g_k$ and using linearity we obtain:
\begin{align*}
    &g_{\pi'}(x_1,\dots,x_l g_k(x_{l+1},\dots,x_{l+k}),x_{l+k+1},\dots,x_n)\\
    &=\sum_{\sigma_1\in NC(V)} C(\sigma,1_V) f_{\pi'}(x_1,\dots,x_l \sum_{\sigma_1\in NC(V)} C(\sigma_1,1_V) f_{\sigma_1}(x_{l+1},\dots,x_{l+k}),x_{l+k+1},\dots,x_n) \\
    &= \sum_{\sigma_1\in NC(V)} C(\sigma,1_V) \sum_{\sigma_2\in NC([n]\backslash V)} C(\sigma_2,\pi') f_{\sigma_2}(x_1,\dots,x_l f_{\sigma_1}(x_{l+1},\dots,x_{l+k}),x_{l+k+1},\dots,x_n).
\end{align*}
Then we construct $\sigma\in NC(n)$ as the unique partition such that $\sigma\leq \{V, [n]\backslash V\}$, $\sigma|V=\sigma_1$ and $\sigma|([n]\backslash V)=\sigma_2$. Thus we get
\begin{align*}
& \sum_{\sigma_1\in NC(V)}\sum_{\sigma_2\in NC([n]\backslash V)} C(\sigma_1,1_V)  C(\sigma_2,\pi') f_{\sigma_2}(x_1,\dots,x_l f_{\sigma_1}(x_{l+1},\dots,x_{l+k}),x_{l+k+1},\dots,x_n) \\
    &= \sum_{\substack{\sigma\in NC(n)\\ \sigma\leq\pi}} \prod_{V\in\pi} C(\sigma|V) f_{\sigma}(x_1,\dots,x_n) \\
    &= \sum_{\substack{\sigma\in NC(n)\\ \sigma\leq \pi}} C(\sigma,\pi)f_{\sigma}(x_1,\dots,x_n) \\
    &= g_\pi(x_1,\dots,x_n),
\end{align*}
as desired.
\end{proof}

\begin{prop}
\label{Prop.relation.cumulants}
Suppose that $(\mathcal{A},\mathcal{B},E)$ is an OV probability space, and let $\{ r^{\mathcal{B}}_n\}_{n \ge 1}$, $\{ \beta^{\mathcal{B}}_n\}_{n \ge 1}$, and $\{ h^{\mathcal{B}}_n\}_{n \ge 1}$ be the families of OV cumulants. Then, the following relations among them hold:
\begin{eqnarray}
	\beta^{\mathcal{B}}_n (a_1, \dots, a_n) 
		&=& \sum_{\pi\in \mathcal{NC}_{irr}(n)} \frac{1}{\tau(\pi)!}h^{\mathcal{B}}_\pi(a_1, \dots, a_n), \label{OVBooMon} \\
	r^{\mathcal{B}}_n (a_1, \dots, a_n) 
		&=& \sum_{\pi\in \mathcal{NC}_{irr}(n)} \frac{(-1)^{|\pi|-1}}{\tau(\pi)!} h^{\mathcal{B}}_\pi(a_1, \dots, a_n), \label{OVFreeMon}
\end{eqnarray}
for every $n$ and elements $a_1,\dots, a_n\in \mathcal{A}$.
\end{prop}

\begin{proof}
To prove \eqref{OVBooMon} we take $\pi\in \mathcal{NC}$ and define: 
$$
	\alpha_\pi (a_1, \dots, a_n):= \sum_{\substack{\sigma\in \mathcal{NC}(n)\\ \sigma \ll\pi}} \left( \prod_{V\in \pi} \frac{1}{\tau(\sigma|V)!}\right) h^{\mathcal{B}}_\sigma(a_1, \dots, a_n).
$$
Using Lemma \ref{Lemma.useful} with $c(\pi)=\frac{1}{\tau(\pi)!} \delta_{\pi\in NC_{irr}}$ we get that $\{\alpha_\sigma\}_{\sigma\in NC}$ is a multiplicative family. Observe that for $\pi\in \mathcal{I}$ and any $\sigma \ll \pi$ one has the nice formula
$\prod_{V\in \pi} \frac{1}{\tau(\sigma|V)!}= \frac{1}{\tau(\sigma)!}.$

Then we can use the OV monotone moment-cumulant formula \eqref{mom-free} to compute:
\begin{equation}
	\widetilde{E}_n =\sum_{\sigma \in \mathcal{NC}(n)} \frac{h^{\mathcal{B}}_\sigma}{\tau(\sigma)!}  =\sum_{\pi \in \mathcal{I}(n)} \sum_{\substack{\sigma \in \mathcal{NC}(n)\\  \sigma\ll\pi}} \frac{h^{\mathcal{B}}_\sigma}{\tau(\sigma)!}=\sum_{\pi \in \mathcal{I}(n)} c_{\pi}.
\end{equation}
Therefore, the family $\{\alpha_\sigma\}_{\sigma\in NC}$ satisfies the OV Boolean moment-cumulant formula \eqref{mom-boo} for all $n$ and elements $a_1,\dots, a_n\in\mathcal{A}$. Thus, we conclude
$$
\beta^{\mathcal{B}}_n (a_1, \dots, a_n)=\alpha_{1_n} (a_1, \dots, a_n)=  \sum_{\substack{\sigma\in \mathcal{NC}(n)\\ \sigma \ll\pi}} \frac{1}{\tau(\sigma)!} h^{\mathcal{B}}_\sigma(a_1, \dots, a_n).
$$

To get \eqref{OVFreeMon} we are going to use the fact the previous equalities depend on the combinatorial structure. More precisely we are going to use that this relation is known in the scalar case. We know that the multiplicative families $\{\beta^{\mathcal{B}}_\sigma\}_{\sigma\in NC}$ and $\{h^{\mathcal{B}}_\sigma\}_{\sigma\in NC}$ are related via the following formula:
$$\beta^{\mathcal{B}}_\pi (a) = \sum_{\substack{\sigma\in \mathcal{NC}(n)\\ \sigma\ll \pi}}  \left( \prod_{V\in \pi} \frac{1}{\tau(\sigma|V)!}\right)  h^{\mathcal{B}}_\sigma(a).$$ 
Using formula \eqref{OVFreeBoo}, we get 
\begin{align*}
	r^{\mathcal{B}}_n(a) 
	&= \sum_{\pi\in \mathcal{NC}_{irr}(n)} (-1)^{|\pi|-1} \beta^{\mathcal{B}}_\pi\\
	&=\sum_{\substack{ \pi\in \mathcal{NC}(n)\\ \pi\ll 1_n }} (-1)^{|\pi|-1}  \sum_{\substack{\sigma\in \mathcal{NC}(n)\\ \sigma\ll \pi}}  \left( \prod_{V\in \pi} \frac{1}{\tau(\sigma|V)!}\right)  h^{\mathcal{B}}_\sigma(a)\\
	&=\sum_{\substack{ \sigma\in \mathcal{NC}(n)\\ \sigma\ll 1_n }}
	h^{\mathcal{B}}_\sigma(a) \sum_{\substack{ \pi\in \mathcal{NC}(n)\\ 1_n\gg\pi\gg\sigma}}  (-1)^{|\pi|-1} \left( \prod_{V\in \pi} \frac{1}{\tau(\sigma|V)!}\right).
\end{align*}

In order to conclude we need to check that the following holds,
\begin{equation}
\label{comb.ident}
\sum_{\substack{ \pi\in \mathcal{NC}(n)\\ 1_n\gg\pi\gg\sigma}}  (-1)^{|\pi|-1} \left( \prod_{V\in \pi} \frac{1}{\tau(\sigma|V)!}\right)= \frac{(-1)^{|\sigma|-1}}{\tau(\sigma)!}, \qquad \forall \sigma\ll 1_n.
\end{equation}
Although this combinatorial equation is highly not trvial, we already know that this equality holds, since we get the same in the scalar valued case. Namely, thanks to \cite{AHLV} we already know that the formulas \eqref{OVFreeBoo}, \eqref{OVBooMon} and \eqref{OVFreeMon} hold in the scalar case $\mathcal{B}=\mathbb{C}$. Thus by our previous steps (and \eqref{OVFreeMon} in the scalar case) we get
$$\sum_{\substack{ \sigma\in \mathcal{NC}(n)\\ \sigma\ll 1_n }} \frac{(-1)^{|\sigma|-1}}{\tau(\sigma)!}
	h^{\mathbb{C}}_\sigma(a) =\sum_{\substack{ \sigma\in \mathcal{NC}(n)\\ \sigma\ll 1_n }}
	h^{\mathbb{C}}_\sigma(a) \sum_{\substack{ \pi\in \mathcal{NC}(n)\\ 1_n\gg\pi\gg\sigma}}  (-1)^{|\pi|-1} \left( \prod_{V\in \pi} \frac{1}{\tau(\sigma|V)!}\right), \qquad \forall a_1,\dots,a_n\in \mathcal{A}.$$
Since this holds for every $a_1,\dots,a_n\in \mathcal{A}$, by varying them we obtain \eqref{comb.ident}. Thus, returning to the OV case, we conclude that \eqref{OVFreeMon} holds.
\end{proof}

Once we have these two formulas relating Boolean and free with monotone cumulants we can invert them using the main result of \cite{celestino2020cumulant}.

\begin{defn}
Let $\pi$ be an irreducible non-crossing partition. We will write $\omega_k(\pi)$ for the number of increasing $k$-colored non-crossing partitions. That is, the number of ways one can decorate the blocks with elements of the set $[k]$ in such a way that if the block $V$ is nested inside the block $W$ this implies $f(V)<f(W)$.  We also define
\begin{equation}
\label{eq.muruaomega}
	\omega(\pi):=\sum\limits_{k=1}^n\frac{(-1)^{k+1}}{k}\omega_k(\pi).
\end{equation}
\end{defn}

\begin{cor}
\label{cor.OV.mon.boofree}
Suppose that $(\mathcal{A},\mathcal{B},E)$ is an OV probability space, and let $\{ r^{\mathcal{B}}_n\}_{n \ge 1}$, $\{ \beta^{\mathcal{B}}_n\}_{n \ge 1}$, and $\{ h^{\mathcal{B}}_n\}_{n \ge 1}$ be the families of OV cumulants. Then, the following relations among them hold:
\begin{eqnarray}
	h^{\mathcal{B}}_n (a_1, \dots, a_n) 
		&=& \sum_{\pi\in \mathcal{NC}_{irr}(n)} \omega(\pi) \beta^{\mathcal{B}}_\pi(a_1, \dots, a_n), \label{OVMonBoo} \\
	h^{\mathcal{B}}_n (a_1, \dots, a_n) 
		&=& \sum_{\pi\in \mathcal{NC}_{irr}(n)} (-1)^{|\pi|-1}\omega(\pi) r^{\mathcal{B}}_\pi(a_1, \dots, a_n), \label{OVMonFree}
\end{eqnarray}
for every $n$ and elements $a_1,\dots, a_n\in \mathcal{A}$.
\end{cor}

We now have all the six formulas relating free, Boolean and monotone cumulants in the OV setting. To prove Theorem \ref{Thm.relation.cumulants}, we can directly translate previous formulas to the analogue formulas in the OVI setting.

\begin{proof}[Proof of Theorem \ref{Thm.relation.cumulants}]
To obtain the six formulas relating OVI cumulants, the general idea is to consider the upper triangular operator-valued probability space $(\widetilde{\mathcal{A}},\widetilde{\mathcal{B}},\widetilde{E})$, apply the corresponding formula in the OV setting, and reading the $(1,2)$-entry when writing the proposition in terms of our original $(\mathcal{A},\mathcal{B},E,E')$ space.  For instance, let us see how to write monotone cumulants in terms of Boolean cumulants (see \eqref{OVIMonBoo}). Fix $n\in\mathbb{n}$ and $a_1,\dots,a_n\in\mathcal{A}$, and  for $i=1,\dots,n$ consider $A_i:=a_iI_2\in \widetilde{\mathcal{A}}$ the $2\times 2$ diagonal matrix with both diagonal entries equal to $a_i$. Then in the upper triangular operator-valued probability space $(\widetilde{\mathcal{A}},\widetilde{\mathcal{B}},\widetilde{E})$ we can apply the monotone-Boolean formula in the OV case (see \eqref{OVMonBoo} from  Corollary \ref{cor.OV.mon.boofree}) to obtain
\[
 h^{\widetilde{\mathcal{B}}}_n  (A_1,\dots,A_n)= \sum_{\pi\in \mathcal{NC}_{irr}(n)} \omega(\pi) \beta^{\widetilde{\mathcal{B}}}_\pi (A_1,\dots,A_n),
\]
and using Lemma \ref{lemma10.Tseng} this is equivalent to
\[
 \begin{bmatrix} h^{\mathcal{B}}_n(a_1,\dots,a_n) & \partial  h^{\mathcal{B}}_n(a_1,\dots,a_n) \\ 0 &  h^{\mathcal{B}}_n(a_1,\dots,a_n)  \end{bmatrix} =\sum_{\pi\in \mathcal{NC}_{irr}(n)} \omega(\pi)  \begin{bmatrix} \beta^{\mathcal{B}}_\pi(a_1,\dots,a_n) & \partial  \beta^{\mathcal{B}}_\pi(a_1,\dots,a_n) \\ 0 &  \beta^{\mathcal{B}}_\pi(a_1,\dots,a_n)  \end{bmatrix}.
\]
Reading the $(1,2)$-entry we obtain \eqref{OVIMonBoo}. Similarly to obtain the monotone-free cumulant (also in \eqref{OVIMonFree}) we can just take the corresponding OV formula \eqref{OVMonBoo} from Corollary \ref{cor.OV.mon.boofree}.

With the same procedure we can use \eqref{OVBooFree} and \eqref{OVFreeBoo} from Proposition \ref{Prop.relation.boofree.cumulants} to get the two equations in \eqref{OVIBooFree}. And we can use \eqref{OVBooMon} and \eqref{OVFreeMon} from Proposition \ref{Prop.relation.cumulants} to obtain the two equations in \eqref{OVIBooMon}.

\end{proof}

\begin{rem}
The relation among free and Boolean cumulants enables to use Boolean cumulants in the study free infinite divisibility via the Boolean Bercovici–Pata bijection. This map assigns to each distribution $\mu$ another distribution $\mathbb{B}(\mu)$, such that the free cumulants of $\mathbb{B}(\mu)$ concide with the Boolean cumulants of $\mu$. Furthermore, it is possible to define maps that continuously interpolate between $\mu$ and $\mathbb{B}(\mu)$ (see the work of Belinschi and Nica \cite{BN1,BN2,BN3}). The operator-valued version of these results was obtained in \cite{ABFN} and the combinatorial results rely on Proposition \ref{Prop.relation.boofree.cumulants}. Once we know how to go from OV to OVI by using the upper triangular probability space associated to each infinitesimal probability space we can translate various results from \cite{ABFN} to OVI setting on the algebraic level. This would follow the same lines on how the infinitesimal scalar case can be obtained from the scalar case as done in \cite{celestino2019relations}.
\end{rem}

\section{Differentiable Paths}
\label{sec:differentiable.paths}

In this section we use differentiable paths to compute the infinitesimal Boolean convolution and infinitesimal monotone convolution which was obtained in Sections 3.1 and 4.1. We will follow the ideas of Section $4.2$ of \cite{tseng2019operator}, where  differentiable paths were used to compute the infinitesimal free convolution.

Suppose that $\mathcal{B}$ is a unital $C^*$-algebra, and $\mathcal{B}\langle X \rangle$ is the $*$-algebra of non-commutative polynomials over $\mathcal{B}$ such that $\mathcal{B}$ and $X$ are algebraically free with $X=X^*.$ 
$\mu$ is called a \emph{$\mathcal{B}$-valued distribution} if $\mu$ is a linear $\mathcal{B}$-bimodule completely positive map $\mu:\mathcal{B}\langle X\rangle \to \mathcal{B}$ with $\mu(1)=1$. A $\mathcal{B}$-distribution $\mu$ is called \emph{exponentially bounded} if there is $M>0$ such that for all $b_1,\dots,b_n\in\mathcal{B}$, we have
$$
\| Xb_1Xb_2\cdots b_nX \| \leq M^{n+1} \|b_1\|\|b_2\|\cdots \|b_n\|. 
$$
We let $\Sigma$ be the space of all $\mathcal{B}$-valued distributions, and $\Sigma_0$ be the space of all $\mathcal{B}$-valued exponentially bounded distributions. 

One of the main reason to consider $\Sigma_0$ instead of $\Sigma$ is due to the fact that for each $\mu\in\Sigma_0$ there exists a $C^*$-operator-valued probability space $(\mathcal{A},\mathcal{B},E)$ and $x=x^*\in\mathcal{A}$ such that $\mu_x=\mu$ (see \cite{PoVi}). Since $G_x$ only depends on $\mu$, here we will define $G_{\mu}(b):=G_{x}(b)$. Then we let $F_{\mu}(b)=G_{\mu}(b)^{-1}$ and also $B_{\mu}(b)=b-F_{\mu}(b).$  

\begin{defn}
A path $\mu(t):[0,1]\to \Sigma_0$ is said to be \emph{differentiable} if for each $n\in\mathbb{N}$, $G^{(n)}_{\mu(t)}$ is differentiable on $H^+(M_n(\mathcal{B}))$ in the sense that for each $t\in [0,1]$ there is a map $Y^{(n)}_t:H^+(M_n(\mathcal{B}))\to M_n(\mathcal{B})$ such that for all $b\in H^+(\mathcal{B})$, we have
$$
\Big\|\frac{G^{(n)}_{\mu(t+h)}(b\otimes 1_n)-G^{(n)}_{\mu(t)}(b\otimes 1_n)}{h} -Y^{(n)}_t(b\otimes 1_n) \Big\| \longrightarrow 0 \text{ as } h\rightarrow 0.
$$
\end{defn}
It seems that we shall deal with a tower of functions, but from the analytic aspect, $G_{\mu(t)}^{(n)}$ on $H^+(M_n(\mathcal{B}))$ essentially has the same behavior of $G_{\mu(t)}$ on $H^+(\mathcal{B}),$ so we will only focus on the ground level. We denote thr ground level limit $Y_t^{(1)}$ by $\partial G_{\mu(t)}$, and the Fr{\'e}chet derivative of $G_{\mu(t)}$ will be denoted by $G'_{\mu(t)}$. That is, 
$$
\partial G_{\mu(t)}(b)=\lim_{h\to0}\!\frac{G_{\mu(t+h)}(b)-G_{\mu(t)}(b)}{h}
$$ and
$$
G_{\mu(t)}'(b)(\cdot)=\lim_{h\to0}\!\frac{G_{\mu(t)}(b+h\cdot)-G_{\mu(t)}(b)}{h}.
$$
We note that  
\begin{eqnarray*}
\partial F_{\mu(t)}(b)&=& \lim\limits_{h\to 0} \frac{G_{\mu(t+h)}(b)^{-1}-G_{\mu(t)}(b)^{-1}}{h}  \\
                      &=&  \lim\limits_{h\to 0} \frac{G_{\mu(t+h)}(b)^{-1}G_{\mu(t)}(b)G_{\mu(t)}(b)^{-1}-G_{\mu(t+h)}(b)^{-1}G_{\mu(t+h)}(b)G_{\mu(t)}(b)^{-1}}{h} \\
                      &=& \lim\limits_{h\to 0} G_{\mu(t+h)}(b)^{-1} \frac{G_{\mu(t)}(b)-G_{\mu(t+h)}(b)}{h} G_{\mu(t)}(b)^{-1} \\
                      &=& -F_{\mu(t)}(b) \partial G_{\mu(t)}(b) F_{\mu(t)}(b).
\end{eqnarray*}
Similarly, 
$$
F'_{\mu(t)}(b)(\cdot) = -F_{\mu(t)}(b) G'_{\mu(t)}(b)(\cdot) F_{\mu(t)}(b).  
$$
\begin{thm}
\label{thm.diff.paths.main}
Suppose that $\mu(t)$ and $\nu(t)$ are differentiable paths, then for $b\in H^+(\mathcal{B})$ we have 
\begin{align*}
&\partial G_{\mu(t)\uplus \nu(t)}(b) \\ &= (F_{\mu(t)}(b)+F_{\nu(t)}(b)-b)^{-1} \Big(F_{\mu(t)}(b)\partial G_{\mu(t)}(b)F_{\mu(t)}(b) +F_{\nu(t)}(b)\partial G_{\nu(t)}(b)F_{\nu(t)}(b)\Big)   (F_{\mu(t)}(b)+F_{\nu(t)}(b)-b)^{-1}
\end{align*}
and 
\begin{equation*}\label{Dmeq}
\partial G_{\mu(t)\triangleright\nu(t)}(b) = G'_{\mu(t)}(F_{\nu(t)})(-F_{\nu(t)}(b)\partial G_{\nu(t)}(b)F_{\nu(t)}(b))+\partial G_{\mu(t)}(F_{\nu(t)}(b)).
\end{equation*}
\end{thm}

\begin{proof}
Given $b\in H^+(\mathcal{B})$, then we have
$B_{\mu(t)}(b)+B_{\nu(t)}(b) = B_{\mu(t)\uplus \nu(t)}(b)$, which implies 
\begin{equation}\label{deq1}
F_{\mu(t)\uplus \nu(t)}(b)=F_{\mu(t)}(b)+F_{\nu(t)}(b)-b. 
\end{equation}
By differentiating \eqref{deq1} with respect to $t$, we obtain 
\begin{eqnarray*}
& & \partial F_{\mu(t)\uplus \nu(t)}(b)=\partial F_{\mu(t)}(b)+\partial F_{\nu(t)}(b) \\
&\Rightarrow & F_{\mu(t)\uplus \nu(t)}(b)\partial G_{\mu(t)\uplus \nu(t)}(b)F_{\mu(t)\uplus \nu(t)}(b) = F_{\mu(t)}(b)\partial G_{\mu(t)}(b)F_{\mu(t)}(b)+F_{\nu(t)}(b)\partial G_{\nu(t)}(b)F_{\nu(t)}(b) \\
&\Rightarrow & 
\partial G_{\mu(t)\uplus \nu(t)}(b) = F_{\mu(t)\uplus \nu(t)}(b)^{-1} (F_{\mu(t)}(b)\partial G_{\mu(t)}(b)F_{\mu(t)}(b)+F_{\nu(t)}(b)\partial G_{\nu(t)}(b)F_{\nu(t)}(b)) F_{\mu(t)\uplus\nu(t)}(b)^{-1}. 
\end{eqnarray*}
Finally we substitute $F_{\mu(t)}(b)+F_{\nu(t)}(b)-b$ for $F_{\mu(t)\uplus \nu(t)}(b)$ to obtain the desired result. 

Note that for $b\in H^+(\mathcal{B})$, 
\begin{equation}\label{deq2}
F_{\mu(t)\triangleright \nu(t)}(b)=F_{\mu(t)}(F_{\nu(t)}(b)).
\end{equation}
we differentiate \eqref{deq2} with respect to $t$, then 
\begin{equation*}
-F_{\mu(t)\triangleright \nu(t)}(b) \partial G_{\mu(t)\triangleright\nu(t)}(b) F_{\mu(t)\triangleright \nu(t)}(b)  
= \partial F_{\mu(t)}(F_{\nu(t)}(b))+F_{\mu(t)}'(F_{\nu(t)}(b))\partial F_{\nu(t)}(b).
\end{equation*}
Thus,
\begin{eqnarray}\label{deq3}
&&-F_{\mu(t)\triangleright \nu(t)}(b) \partial G_{\mu(t)\triangleright\nu(t)}(b) F_{\mu(t)\triangleright \nu(t)}(b)  \\
&=&-F_{\mu(t)}(F_{\nu(t)}(b))\partial G_{\mu(t)}(F_{\nu(t)}(b))F_{\mu(t)}(F_{\nu(t)}(b))+F_{\mu(t)}'(F_{\nu(t)}(b))(-F_{\nu(t)}(b)\partial G_{\nu(t)}(b)F_{\nu(t)}(b)). \nonumber 
\end{eqnarray}
Now, we note that 
\begin{equation}\label{deq4}
F_{\mu(t)}'(F_{\nu(t)}(b))(\cdot) = -F_{\mu(t)}(F_{\nu(t)}(b))[G'_{\mu(t)}(F_{\nu(t)}(b))(\cdot)]F_{\mu(t)}(F_{\nu(t)}(b)).
\end{equation}
Applying \eqref{deq2} and \eqref{deq4} on \eqref{deq3}, we get our final result
$$
\partial G_{\mu(t)\triangleright\nu(t)}(b) = G'_{\mu(t)}(F_{\nu(t)})(-F_{\nu(t)}(b)\partial G_{\nu(t)}(b)F_{\nu(t)}(b))+\partial G_{\mu(t)}(F_{\nu(t)}(b)).
$$
\end{proof}

\appendix
\addtocontents{toc}{\protect\setcounter{tocdepth}{0}}

\section{Anti Trace Models}
\label{sec:antitrace.model}

Here we study the infinitesimal distribution of the random matrix model presented in Section 5 of \cite{MLE}. 
Let $\mathcal{M}_N(\mathbb{C})$ be the set of $N\times N$ complex-valued matrices and let $C_N \in \mathcal{M}_N(\mathbb{C})$ be the matrix with all of its entries equal to $\frac{1}{N}$. We are going to consider the functional $\Psi_N: \mathcal{M}_N(\mathbb{C})\to \mathbb{C}$ that maps $M$ to $\Psi_N(M)=\mathbb{E}( Tr(M C_N)).$ In other words, if $M=(M_{ij})_{1\leq i,j\leq n}$ then 
$$\Psi_N(M)=\frac{1}{N}\sum_{i,j\in [N]}\mathbb{E}(M_{ij}).$$

The random matrices $A_N$ that we want to study are constructed as follows. Consider $X_1,\dots, X_N$ complex Gaussian random variables and create the $N\times N$ random matrix
$$A_N=\left( \frac{1}{N}(X_i+\overline{X_j})\right)_{i,j=1}^N.$$

From \cite{MLE} we know that the families of matrices $A_N$ and $C_N$ are asymptotically Boolean independent with respect to $\Psi_N$. Moreover, the families of matrices $A_N$ and $A^T_N$ (where $T$ stands for the transpose) are asymptotically Boolean independent with respect to $\Psi_N$.

A very natural question is whether or not this result can be extended to the infinitesimal setting.
The main purpose of this section is to compute not only the limit of $\Psi_N(M_N^k)$ but also the limit $N(\Psi_N(M_N^k)-1)$ when $N\to \infty$, in order to check if the pairs of variables are asymptotically infinitesimal Boolean independent. The main result is the following:

\begin{thm}
\label{thm.1.appendix.A}
Let us assume that $(A_N)_{N\geq 1}$, $(A^T_N)_{N\geq 1}$ and $(C_N)_{N\geq 1}$ converge in distribution to the random variables $a,\bar{a}$ and $c$, respectively. then
\begin{itemize}
    \item $a$ and $b$ are not infinitesimal Boolean independent.
    \item $a$ and $c$ are not infinitesimal Boolean independent.
\end{itemize}
\end{thm}

Let us also say that $(A_N+A^T_N)_{N\geq 1}$  converge in distribution to the random variable $b$. In order to get the result, we are going to compute the moments and infinitesimal moments of $a$, $b$ and $c$. Notice that $a$ and $\bar{a}$ have the same distributions, so we will directly obtain the moments of $\bar{a}$. For $c$ the computation is very simple.

\begin{rem}
Since $\Psi_N(C_N^k)=\mathbb{E}(C_N^{k+1})=1$ for all $N$, the limiting moments of $C_N$ are $m_k(c)=1$, and the infinitesimal moments are $m'_k(C)=0$ for all $k$. Thus we easily get its Boolean and infinitesimal Boolean cumulants:
$$\beta_n(c)=\delta_{n=1}, \qquad \beta'_n(c)=0, \qquad \forall n\in\mathbb{N},$$
\end{rem}

The moments and infinitesimal moments of $a$ are recorded in the following proposition. 

\begin{prop}
\label{prop.antitrace.lim}
With the previous notation we have the following
\begin{equation}
\label{eq.antitrace.lim}
m_k(a)=\lim_{N\to\infty} \Psi_N(A_N^k)=\begin{cases} 1 & \text{if }k\text{ even,} \\ 0 & \text{if }k\text{ odd}  \end{cases},
\end{equation}

\begin{equation}
\label{eq.antitrace.inf.lim}
m'_k(a)=\lim_{N\to\infty} N(\Psi_N(A_N^k)-1)=\begin{cases} \frac{1}{8}(3k^2-2k) & \text{if }k\text{ even,} \\ 0 & \text{if }k\text{ odd}  \end{cases}.
\end{equation}
\end{prop}

\begin{rem}
The moments $(m_k)_{k\geq 1}$ were already computed in \cite{MLE}, but we recover the result here as an intermediate step while computing $m'_k$.
\end{rem}

\begin{proof}
Let $M_N=N A_N$. We start by computing

\begin{equation}
\label{eq.1.app.bool.trace}
\mathbb{E}(Tr(M_N^kC_N))= \mathbb{E} \left( \sum_{j\in [N]^{k+1}}\prod_{r=1}^{k}  (X_{j(r)}+\overline{X}_{j(r+1)})  \right)=  \sum_{j\in [N]^{k+1}} \sum_{\epsilon\in[0,1]^k} \mathbb{E} \left(\prod_{r=1}^{k}  X^{(\epsilon(r))}_{j(r+\epsilon(r))}  \right),
\end{equation}

where we use the vector $j=(j(1),\dots,j(k+1))\in [N]^{k+1}$ to encode the possible combination of indices in the product $M_N^kJ_N$ and we use the vector $\epsilon\in[0,1]^k$ and the notation notation $X^{(0)}:=X$ and $X^{(1)}:=\overline{X}$ to encode the $2^k$ terms obtained form the product of the entries. Then for every $l\in[N]$ we take $V_l:=\{i\in[k+1]:j(i)=l\}$ (possibly empty) and to vector $j$ we associate the partition $\ker(j):=\{V_1,V_2,\dots,V_N\}\in \mathcal{P}(k+1)$. We observe that if $j,l:[N]^{k+1}$ both generate the same partition $\ker(j)=\ker(l)=\pi$, then 
$$E_\pi:=\sum_{\epsilon\in[0,1]^k} \mathbb{E} \left( \prod_{r=1}^{k}  X^{(\epsilon(r))}_{j(r+\epsilon(r))}  \right) =\sum_{\epsilon\in[0,1]^k} \mathbb{E} \left(\prod_{r=1}^{k}  X^{(\epsilon(r))}_{l(r+\epsilon(r))}  \right),$$
since the random variables are independent and identically distributed.

Also, we observe that a fixed partition $\pi$ can be obtained from $(N)_{|\pi|}:=N(N-1)\dots(N+1-|\pi|)$ different vectors $j:[N]^{k+1}$, where $|\pi|$ is the number of blocks of $\pi$. 
Thus using \eqref{eq.1.app.bool.trace} we get
\[
\Psi_N(A_N^k)=\frac{\mathbb{E}(Tr(M_N^kC_N))}{N^{k+1}} =\frac{1}{N^{k+1}} \sum_{\pi\in \mathcal{P}(k+1)} \sum_{\substack{j:[k+1]\to [N] \\  \ker(j)=\pi}}  E_{\pi}=\frac{1}{N^{k+1}} \sum_{\pi\in \mathcal{P}(k+1)} (N)_{|\pi|}  E_\pi.
\]

Letting $N$ tend to infinity, we get that $m_k(a)$ and $m'_k(a)$ are the coefficients of $N^k$ and $N^{k+1}$, respectively, in the polynomial:
$$\sum_{\pi\in \mathcal{P}(k+1)} (N)_{|\pi|}  E_\pi.$$

Then we only need to consider partitions $\pi$ with $|\pi|\geq k$. That is, either $\pi=1_{k+1}$ or the partitions $\lambda_{x,y}=\{\{x,y\},\{1\},\dots,\{k+1\}\}$ with one pairing $\{x,y\}$ and $k-1$ singletons for $1\leq x\leq y\leq k+1$. Recall that the $X_i$ are complex Gaussian, so $\mathbb{E}(X_i^p \bar{X}_j^q)= p!$ if $p=q$ and 0 otherwise. This implies that $\mathbb{E}\left( \prod_{r=1}^{k}  X^{(\epsilon(r))}_{j(r+\epsilon(r))}  \right)=0$ whenever $k$ is odd (since $X_j$ and $\overline{X_j}$ should come in pairs). Therefore we conclude that $m_k=m'_k=0$ whenever $k$ is odd, and we now assume $k=2l$. The $X_j$ being Gaussian and independent also implies that whenever there exists and $r\in[k]$ such that for every $s\in[k]$ with $s\neq r$ we have $j(r+\epsilon(r))\neq j(s+\epsilon(s))$, then  $\mathbb{E}\left( \prod_{r=1}^{k}  X^{(\epsilon(r))}_{j(r+\epsilon(r))}  \right)=0$.

For the case $\pi=1_{k+1}$, all the entries of $j$ are different and thus $j(r+\epsilon(r))\neq j(s+\epsilon(s))$ whenever $r+\epsilon(r)\neq s+\epsilon(s)$, which is true for all $0\leq r < s\leq k-1$, except when $r+1=s$, $\epsilon(s)=0$ and $\epsilon(r)=1$. This in particular forces us to couple $\overline{X}_{j(2)}$ with $X_{j(2)}$ by taking $\epsilon(1)=1$ and $\epsilon(2)=0$, which in turn forces us to couple $\overline{X}_{j(4)}$ with $X_{j(4)}$ by taking $\epsilon(3)=1$ and $\epsilon(4)=0$, and so on.  Therefore, the only $\epsilon\in[0,1]^k$ that will contribute to the sum $E_{1_{2l}}$ is the one that is alternating and begins with 1, namely $\epsilon_{\mbox{alt}}(1,0,1,0,\dots,1,0)$. For the rest of the $\epsilon$ we will get at least $l+1$ different values for $j(r+\epsilon(r))$ where $r=0,\dots,k-1$ and thus we can find an $r$ such that $j(r+\epsilon(r))$ is different from the other values. Thus, we get that
$E_{1_{k+1}}=\prod_{r=1}^l \mathbb{E}\left(  X_{j(2r)}\bar{X}_{j(2r)} \right)=1$ 
and we conclude that the coefficient of $N^{k+1}$ in $(N)_{k+1}E_{1_k}$ is 1, and we recover \eqref{eq.antitrace.lim}. On the other hand the coefficient of $N^{k}$ in $(N)_{k+1}E_{1_k}$ is $-\binom{k+1}{2}=-l(2l+1)$.

For the case $\pi=\lambda_{x,y}$, almost all entries of $j$ are different (except $j(x)=j(y)$). Therefore, the vectors $\epsilon$ that will not vanish are either $\epsilon_{\mbox{alt}}$ or almost alternating, namely of the form
\[
(1,0,\dots,1,0,*,1,0,\dots,1,0,*,1,0,\dots,1,0),
\]
where $*$ are in entries $r\leq s$ (with $r$ odd and $s$ even). since we need the equality $j(r+\epsilon(r))= j(s+\epsilon(s))$ we require to have $r+\epsilon(r)=x$, and $s+\epsilon(s)=y$, with $\epsilon(r)\neq \epsilon(s)$. So the only chances are $(r,\epsilon(r),s,\epsilon(s))=(x-1,1,y,0)$ or $(r,\epsilon(r),s,\epsilon(s))=(x,0,y-1,1)$. 

Then we arrive to the following cases 
\begin{itemize}
    \item $x$ and $y$ are both even. There are  $\binom{l}{2}$ pairs $(x,y)$ of this form, and the vector 
    \[
    (1,0,\dots,1,0,1,1,0,\dots,1,0,0,1,0,\dots,1,0)
    \]
    will contribute with 1, while $\epsilon_{\mbox{alt}}$ will contribute with 2 because we get a $\mathbb{E}(X_{j(x)}\overline{X_{j(x)}}X_{j(y)}\overline{X_{j(y)}})=2$. Thus from this case we get a contribution of $3\frac{l^2-l}{2}$.
    
    \item $x$ and $y$ are both odd. There are  $\binom{l+1}{2}$ pairs $(x,y)$ of this form, and the vector 
    \[
    (1,0,\dots,1,0,0,1,0,\dots,1,0,1,1,0,\dots,1,0)
    \]
    will contribute with 1, while $\epsilon_{\mbox{alt}}$ will contribute with 1. Thus we get a contribution of 
    $l^2+l$.
    
    \item $x$ and $y$ have different parity. There are $l(l+1)$ pairs $(x,y)$ of this form, and only $\epsilon_{\mbox{alt}}$ will contribute with $1$. Thus we get a contribution of 
    $l^2+l$
\end{itemize}

Summing up all contributions to the second coefficient we get the desired value 
$$m'_k(a)=(-2l^2-l)+3\frac{l^2-l}{2}+
(l^2+l)+l^2+l=\frac{3l^2-l}{2}=\frac{1}{8}(3k^2-2k)$$
\end{proof}

Once we know the moments of $a$, it is easy to compute its Boolean cumulants.

\begin{lem} 
Let $\{\beta_n(a)\}_{n\geq 1}$ and $\{\beta'_n(a)\}_{n\geq 1}$ be the Boolean and infintesimal Boolean cumulants of $a$, respectively. Then
$$\beta_n(a)=\begin{cases} 1 & \text{if }n=2 \\ 0 & \text{otherwise}\end{cases} \qquad \mbox{and} \qquad  \beta'_n(a)=\begin{cases}1 & \text{if }n=2, \\ 3 & \text{if }n\neq 2\text{ even} \\ 0 & \text{if }n\text{ odd.}  \end{cases}$$
\end{lem}

\begin{proof}
We just need to check that cumulants fulfill the moment-cumulant formulas. We will omit the valuation on $(a)$ as we just work with this variable. Indeed, with these values for the $\beta_k$ we have that the product $\beta_\pi$ for an interval partition $\pi$ is zero, unless $\pi$ has all its blocks of size 2. Since there only exists one such interval partition when $n$ is even, and none when $n$ is odd we get 
$$m_n=\sum_{\pi\in \mathcal{I}(n)} \beta_\pi.$$
For the infinitesimal case we get $\partial \beta_\pi=\sum_{V\in\pi} \beta'_{|V|} \prod_{V\neq W\in\pi} b_{|W|}=0$ unless all the blocks of $\pi$ but one have size 2. For the other block, say $V$, we need $\beta'_{|V|}\neq 0$, so $V$ must have even size. Then for $n$ odd the moment-cumulant formula clearly holds (both sides are 0), while for $n$ even case we denote by $\tau=\{(1,2),(3,4),\dots, (n-1,n)\}$ the unique interval partition with all the blocks of size 2, and denote by $\mathcal{I}_2(n)$ the set of partitions with all the blocks of size 2 except one strictly different. Observe that number of partitions $\pi\in \mathcal{I}_2(n)$ with size $|\pi|=j$ is $j$, for $j=1,\dots,\frac{n}{2}-1$, since the size of the disinct block should be $n-2j$, thus we just need to choose among the $j$ possible positions. Then $|\mathcal{I}_2(n)|=\frac{1}{2}\left(\frac{n}{2}-1\right)\frac{n}{2}=\frac{n^2-2n}{8}$, and we conclude that 
$$\sum_{\pi\in \mathcal{I}(n)} \partial \beta_\pi =\beta_\tau +\sum_{\pi\in \mathcal{I}_2(n)} \partial \beta_\pi=\frac{n}{2} +\sum_{\pi\in \mathcal{I}_2(n)} 3= \frac{n}{2} +3\frac{n^2-2}{8}= \frac{3n^2-2}{8}=m'_n.$$
\end{proof}

Now we compute the moments and infinitesimal moments of $b$:

\begin{prop}
\label{prop.antitrace.moments.b}
The limiting moments and infinitesimal moments of $A_N+A_N^T$ are given by
\[
m_k(b)=\begin{cases} 2^{k/2} & \text{if }k\text{ even,} \\ 0 & \text{if }k\text{ odd}  \end{cases}   \qquad
m'_k(b)=\begin{cases} \frac{3}{4}k^22^{k/2} & \text{if }k\text{ even,} \\ 0 & \text{if }k\text{ odd}  \end{cases}.
\] 
\end{prop}

\begin{proof}
Observe that the $(i,j)$-entry of $A_N+A^T_N$ is $\tfrac{1}{N} (X_i+\bar{X}_i+X_j+\bar{X}_j) =\tfrac{1}{N} (Re(X_i)+Re(X_j))$. We will denote $R_i=2Re(X_i)$. Recall that $R_i$ is a centered real Gaussian with variance 2, so $\mathbb{E}(R_i^p)=2^{p/2}(p-1)!!$ if $p$ even and 0 otherwise. Proceeding as in Proposition \ref{prop.antitrace.lim} we get

\[
\Psi_N((A_N+A^T_N)^k)=\frac{1}{N^{k+1}}\sum_{j\in [N]^{k+1}} \sum_{\epsilon\in[0,1]^k} \mathbb{E} \left(\prod_{r=1}^{k}  R_{j(r+\epsilon(r))}  \right)=\frac{1}{N^{k+1}} \sum_{\pi\in \mathcal{P}(k+1)} (N)_{|\pi|}  F_\pi,
\]

where 
$$F_\pi:=\sum_{\epsilon\in[0,1]^k} \mathbb{E} \left( \prod_{r=1}^{k}  R_{j(r+\epsilon(r))}  \right),$$
is the value we obtain for the sum whenever $\ker(j)=\pi$. We again get that $m_k=m'_k=0$ if $k$ odd and we write $n=2l$ for the even case. For the coefficient of $N^{k+1}$ we just need to focus on $1_{k+1}$ and we directly get
$$m_{2l}=\mathbb{E}(R_{j(2)}^2)\mathbb{E}(R_{j(4)}^2)\dots \mathbb{E}(R_{j(2l)}^2)=2^l.$$

Now for the infinitesimal moment, we look for the coefficient of $N^{2l}$. For $\pi=1_{k+1}$ the coefficient of $N^{2l}$ is $-\binom{k+1}{2}2^l=-l(2l+1)2^l.$

For the case $\pi=\lambda_{x,y}$, now do not have the condition $\epsilon(r)\neq \epsilon(s)$, So there are four possibilities: 
$(r,\epsilon(r),s,\epsilon(s))=\{ (x-1,1,y,0),(x,0,y-1,1), (x,0,y,0),(x-1,1,y-1,1)\}$. 

Then we arrive to the following cases 
\begin{itemize}
    \item $x$ and $y$ are both even. There are  $\binom{l}{2}$ pairs $(x,y)$ of this form, and the vector 
    \[
    (1,0,\dots,1,0,1,1,0,\dots,1,0,0,1,0,\dots,1,0)
    \]
    will contribute with $2^l$, while $\epsilon_{\mbox{alt}}$ will contribute with $2^{l-2}\cdot (4\cdot 3)=2^l\cdot 3$ because we get a $\mathbb{E}(R_{j(x)}^2R_{j(y)}^2)$. Thus from this case we get a contribution of $2^{l+2}\frac{l^2-l}{2}=2^{l+1}(l^2-l) $.
    
    \item $x$ and $y$ are both odd. There are  $\binom{l+1}{2}$ pairs $(x,y)$ of this form, and the vector 
    \[
    (1,0,\dots,1,0,0,1,0,\dots,1,0,1,1,0,\dots,1,0)
    \]
    will contribute with $2^l$, while $\epsilon_{\mbox{alt}}$ will contribute with $2^l$. Thus we get a contribution of 
    $2^{l}(l^2+l)$.
    
    \item $x$ and $y$ have different parity. There are $l(l+1)$ pairs $(x,y)$ of this form, the corresponding vector will contribute with $2^l$, while $\epsilon_{\mbox{alt}}$ will contribute with $2^l$. Thus we get a contribution of $2^{l+1}(l^2+l)$
\end{itemize}

Summing up all contributions to the second coefficient we get the desired value 
$$-l(2l+1)2^l+2^{l+1}(l^2-l)+2^{l}(l^2+l)+2^{l+1}(l^2+l)=2^l\left( -2l^2-l+2l^2-2l+l^2+l+2l^2+2l \right) =3l^22^l.$$
\end{proof}

Now that we have a better understanding of the infinitesimal distributions of $a,b$ and $c$ we can prove the main result.

\begin{proof}[Proof of Theorem \ref{thm.1.appendix.A}]
\begin{itemize}
    \item  For the case of $a$ and $b$, we first observe that the Boolean cumulants satisfy
\[
\beta_n(a)+\beta_n(\bar{a})=2\beta_n(a)=\begin{cases} 2 & \text{if }n=2, \\ 0 & \text{otherwise,}\end{cases}
\]
so $\beta_n(a)+\beta_n(\bar{a})= \beta_n(a+\bar{a})$ and we recover the result that $a$ and $\bar{a}$ are Boolean independent. However, if $a$ and $\bar{a}$ were also infinitesimal Boolean, this would imply that the infinitesimal Boolean cumulants of $b=a+\bar{a}$ are given by
\[
\beta'_n(b)=\beta'_n(a)+\beta'_n(\bar{a})=2\beta'_n(a)=\begin{cases}2 & \text{if }n=2, \\ 6 & \text{if }n\neq 2\text{ even,} \\ 0 & \text{if }n\text{ odd.}  \end{cases}
\]
This in turn means that $\varphi'(b^2)=2$, but we already computed this value and we know that $\varphi'(b^2)=6$, thus we have a contradiction. Therefore, even though that $A_N$ and $B_N^T$ are asymptotically Boolean independent, they are not asymptotically infinitesimal Boolean independent.

\item For the case of $a$ and $c$, we will restrict our attention to computing the mixed moment $\varphi'(a^2ca^2)$. First we notice that 
\[
\Psi_N(A_N^2C_NA_N^2)=\frac{1}{N^6}\sum_{i_1,\dots,i_6\in [N]}\mathbb{E}(A_{i_1i_2}A_{i_2i_3}A_{i_4i_5}A_{i_5i_6}).
\]

And restricting to the case where at most 2 of $i_1,\dots i_6$ coincide, we get that 
\begin{align*}
\mathbb{E}(A_{i_1i_2}A_{i_2i_3}A_{i_4i_5}A_{i_5i_6})&=\mathbb{E}((X_{i_1}+\overline{X}_{i_2})(X_{i_2}+\overline{X}_{i_3})(X_{i_4}+\overline{X}_{i_5})(X_{i_5}+\overline{X}_{i_6})) \\
&= \mathbb{E}((X_{i_2}\overline{X}_{i_2}+X_{i_1}\overline{X}_{i_3}) (X_{i_5}\overline{X}_{i_5}+X_{i_4}\overline{X}_{i_6})),
\end{align*}
where the only non-vanishing cases are when $i_1=i_3$, $i_4=i_6$ or $i_2=i_5$.

We conclude that 
\[
\Psi_N(A_N^2C_NA_N^2)=\frac{N^6+(1+1+1)N^5+o(N^5)}{N^6}=1+3/N+o(1/N)
\]
Therefore, $\varphi'(a^2ca^2)=3$. If $a$ and $c$ were infinitesimally Boolean independent we would have the equality
\[
\varphi'(a^2ca^2)=\varphi'(a^2)\varphi(c)\varphi(a^2)+\varphi(a^2)\varphi'(c)\varphi(a^2)+\varphi(a^2)\varphi(c)\varphi'(a^2).
\]

But we already know that the right hand side is equal to 
\[
1\cdot 1\cdot 1+1\cdot 0\cdot 1+1\cdot 1\cdot 1=2\neq 3.
\]
Thus we get a contradiction and conclude that $a$ and $c$ cannot possibly be infinitesimally Boolean independent.
\end{itemize}

\end{proof}

\begin{rem}
In general, it can be proved that $\varphi'(a^{h_1}c^{l_1}a^{h_2}c^{l_2}\cdots a^{h_s})=0$ whenever $h_j$ is odd for some $j=1,\dots, s$. When all $h_j=2k_j$ are even we have
\begin{equation}
\label{eq.1.appendix.A}
\varphi'(a^{2k_1}c^{l_1}a^{2k_2}c^{l_2}\cdots a^{2k_s})=\binom{k}{2}+k_1^2+\dots+k_s^2,
\end{equation}
where $k:=k_1+\dots+k_s$. Whereas for the sum 
\[
S:=\sum_{j=1}^s \varphi(a^{h_1}) \varphi(c^{l_1})\dots  \varphi(c^{l_{j-1}})\varphi'(a^{h_j}) \varphi(c^{l_j})\cdots  \varphi(a^{h_s})+ \sum_{j=1}^{s-1}\varphi(a^{h_1}) \varphi(c^{l_1})\dots  \varphi(a^{h_{j}})\varphi'(c^{l_j}) \varphi(a^{h_{j+1}})\cdots  \varphi(a^{h_s})
\]

that should be equal to \eqref{eq.1.appendix.A} if $a$ and $c$ were infinitesimal Boolean independent) we get $S=0$ whenever $h_j$ is odd for some $j=1,\dots,s$. And when all $h_j=2k_J$ are even we get

\begin{equation}
S=\frac{3}{2}\left(k_1^2+\dots+k_s^2\right)-\frac{k}{2}.
\end{equation}

Thus $S$ actually coincide with \eqref{eq.1.appendix.A} when some $k_j$ is odd, but for the even case they differ by 
\[
\frac{ k_1^2+\dots+k_s^2-(k_1+\dots+k_s)^2}{2}.
\]

\end{rem}


\section{Partial Trace Models}
\label{sec:partal.trace.model}

In \cite{lenczewski2014limit}, the author introduced the partial traces. In this subsection, we will consider the simple version of the partial trace.  
First, we introduce the following notations. 

For a given $n\times n$ random matrix $A=[a_{i,j}]_{i,j=1}^n$, we write it by four block matrices:
$$
A =
\begin{bmatrix}
A_{1,1} & A_{1,2}  \\
A_{2,1} & A_{2,2} 
\end{bmatrix}
$$
where $A_{1,1}$ is $(n-1) \times (n-1)$ square matrix, $A_{1,2}$ is $(n-1)\times 1$ column matrix, $A_{2,1}$ is $1 \times (n-1)$ matrix, and $A_{2,2}=a_{n,n}$. Then we let 
$$
T_{A} =
\begin{bmatrix}
0  & A_{1,2} \\
A_{2,1} &  0 
\end{bmatrix}
$$ 
and define the \emph{partial trace} $\tau_n$ on $A$ by  
$$
\tau_n (A) = E ( a_{n,n} ).
$$

\textbf{Infinitesimal partial trace of GUE}

Suppose that $A=[a_{i,j}]_{i,j=1}^n$ is a GUE ensemble, and if we write $T_A=[c_{i,j}]_{i,j=1}^n$, then 
$c_{i,n}=a_{i,n}$ for all $1\leq i \leq n-1$, 
$c_{n,j}=a_{n,j}=\overline{a_{j,n}}$ for all $1\leq j\leq n-1$, and  
$c_{i,j}=0$ otherwise. 

Let $k\in\mathbb{N}$, we note that
$$
\tau_n(T_A^k) = \sum\limits_{i_1,\dots,i_{k-1}\in [n]} E(c_{n,i_1}c_{i_1,i_2}\cdots c_{i_{k-1},n}).
$$
It is easy to see that $c_{n,i_1}c_{i_1,i_2}\cdots c_{i_{k-1},n}=0$ expect $k$ is even and $i_2=i_4=i_6=\cdots = i_{k-2}=n$. Let $k=2m$ for some $m\in \mathbb{N}$, we have
\begin{eqnarray*}
\tau_n(T_A^k) &=&\tau_n(T_A^{2m}) \\ 
              &=& \sum\limits_{i_1,i_{2},i_{3},\dots ,i_{2m-2},i_{2m-1}\in [n]} E(c_{n,i_1}c_{i_1,i_2}\cdots c_{i_{2m-2},i_{2m-1}}c_{i_{2m-1},n}) \\
              &=& \sum\limits_{j_1,j_2,\dots,j_{m} \in [n-1]} E(c_{n,j_1}c_{j_1,n}\cdots c_{n,j_{m}}c_{j_{m},n})  \\
              &=& \frac{1}{n^{m}} \sum\limits_{j_1,j_2,\dots,j_m\in [n-1]} E(|x_{n,j_1}|^2|x_{n,j_2}|^2\cdots |x_{n,j_m}|^2)
\end{eqnarray*}
where $x_{i,j}=\sqrt{n}a_{i,j}$ for all $i,j$.  

In order to compute the limit law and infinitesimal law, we need to find the coefficient of $n^m$ and $n^{m-1}$. Hence, we shall calculate the following two terms. 
\begin{eqnarray*}
& & \sum\limits_{j_1,j_2,\dots,j_m \text{ are all different}} E(|x_{n,j_1}|^2|x_{n,j_2}|^2\cdots |x_{n,j_m}|^2) \\
&=& \sum\limits_{j_1,j_2,\dots,j_m \text{ are all different}} E(|x_{n,j_1}|^2)E(|x_{n,j_2}|^2)\cdots E(|x_{n,j_m}|^2) \\
&=& (n-1)(n-2)\cdots (n-m) \\
&=& n^{m}+\frac{-1}{2}m(m+1)n^{m-1}+O(n^{m-1}),
\end{eqnarray*}
and
\begin{eqnarray*}
& &\sum\limits_{ \text{ two of }j_1,j_2,\dots,j_m \text{ are the same}} E(|x_{n,j_1}|^2|x_{n,j_2}|^2\cdots |x_{n,j_m}|^2) \\
&=& \binom{m}{2} (n-1)\cdots (n-m+1) E(|x_{n,j_1}|^4) E(|x_{n,j_3}|^2)\cdots E(|x_{n,j_m}|^2) \\
&=& \frac{m(m-1)}{2} \Big(n^{m-1}+O(n^{m-1}) \Big) 2\\
&=& m(m-1) n^{m-1}+ O(n^{m-1}). 
\end{eqnarray*}
Hence, we have that the coefficient of $n^{m}$ is $1$, and the coefficient of $n^{m-1}$ is 
$$
\frac{-1}{2}m(m+1)+m(m-1) = \frac{1}{2}m^2-\frac{3}{2}m.
$$
Thus, we conclude that the 
$$
m_k(T_A)=\lim\limits_{n\rightarrow \infty}  \tau_n(T_A^k) = \begin{cases}
0  & \text{if }k \text{ is odd} \\
1  & \text{if }k \text{ is even} 
\end{cases},
$$
and 
$$
m'_k(T_A)=\lim\limits_{n\rightarrow \infty} n \Big( \tau_n(T_A^k)- m_k \Big) = \begin{cases}
0 & \text{if }k \text{ is odd} \\
\frac{1}{2}m^2-\frac{3}{2}m & \text{if }k \text{ is even, and } k=2m   
\end{cases}.
$$

\textbf{Infinitesimal Boolean Cumulants of $T_A+T_B$}

Now we consider two independent GUE ensembles $A$ and $B$. It is known that $T_A$ and $T_B$ are asymptotically Boolean with respect to $\tau_n$ ( see Theorem 7.1 of \cite{lenczewski2014limit} ). However, we will show that they are not asymptotically infinitesimally Boolean w.r.t $\tau_n$.
In order to do this, we shall compute the infinitesimal Boolean cumulants of $T_A$ and $T_A+T_B$. 
First, we compute the first four infinitesimal moments of $T_A+T_B$. Observe that
\begin{equation}\label{bpeqn}
\tau_n(T_A^{s_1}T_B^{t_1}T_A^{s_2}T_B^{t_2}\cdots T_A^{s_m}T_B^{t_m}) =0 
\end{equation}
whenever $s_1+\cdots +s_m$ is odd or $t_1+\cdots+t_m$ is odd. Hence, all odd moments and infinitesimal odd moments vanish. Also, 
$\tau_n(T_AT_B)=\tau_n(T_BT_A)=0$ for all $n\in\mathbb{N}$, which provides us 
$$
\tau_n ( (T_A+T_B)^2 ) = \tau_n(T_A^2)+\tau_n(T_B^2)+\tau_n(T_AT_B)+\tau_n(T_BT_A) \longrightarrow 1+1+0+0 =2 \text{ as } n\rightarrow \infty.
$$
Then,   
$$
n \Big(\tau_n( (T_A^2+T_B^2)-2 ) \Big) = n \Big( \tau_n(T_A^2)-1 \Big) + n \Big( \tau_n(T_B^2)-1 \Big) \longrightarrow -2 \text{ as } n\rightarrow \infty.
$$
For the fourth moment, by \eqref{bpeqn} and the symmetry of $A$ and $B$, we have that 
\begin{equation*}
\tau_n( (T_A+T_B)^4 ) = 2\tau_n(T_A^4) + 2\tau_n(T_A^2T_B^2) + 2\tau_n(T_AT_B^2T_A) + 2\tau_n(T_AT_BT_AT_B).
\end{equation*}
Suppose that $T_A=[\frac{1}{\sqrt{n}}x_{i,j}]$, $T_B=[\frac{1}{\sqrt{n}}y_{i,j}]$. Then 
\begin{equation*}
\tau_n(T_AT_BT_AT_B) = \frac{1}{n^2} \sum\limits_{i_1,i_2\in [n-1]} E(x_{n,i_1}y_{i_1,n}x_{n,i_2}y_{i_2,n}) 
                     = \frac{1}{n^2} \sum\limits_{i_1,i_2\in [n-1]} E(x_{n,i_1}x_{n,i_2})E(y_{i_1,n}y_{i_2,n}) 
                     = 0
\end{equation*}
and 
\begin{eqnarray*}
\tau_n(T_A^2T_B^2) &=& \frac{1}{n^2}\sum\limits_{i_1,i_2\in [n-1]} E(x_{n,i_1}x_{i_1,n}y_{n,i_2}y_{i_2,n}) 
                   =\frac{1}{n^2}\sum\limits_{i_1,i_2\in [n-1]} E(|x_{n,i_1}|^2) E(|y_{n,i_2}|^2) \\
                   &=& \frac{1}{n^2}\sum\limits_{i_1,i_2\in [n-1]}1 
                   = (1-\frac{1}{n})^2=1-\frac{2}{n}+\frac{1}{n^2}.
\end{eqnarray*}
In addition,
\begin{eqnarray*}
\tau_n(T_AT_B^2T_A)&=& \frac{1}{n^2}\sum\limits_{i_1,i_2\in [n-1]} E(x_{n,i_1}y_{i_1,n}y_{n,i_2}x_{i_2,n}) 
                   =\frac{1}{n^2}\sum\limits_{i_1,i_2\in [n-1]} E(x_{n,i_1}x_{i_2n})E(y_{i_1,n}y_{n,i_2})\\ 
                   &=& \frac{1}{n^2}\sum\limits_{i\in [n-1]}1 
                   = \frac{1}{n}-\frac{1}{n^2}.
\end{eqnarray*}
Thus, we conclude that 
$$
\tau_n((T_A+T_B)^4) = 4 +\frac{-4}{n}+O(\frac{1}{n}),
$$
which implies that 
$$
\lim\limits_{n\rightarrow \infty} \tau_n((T_A+T_B)^4) =4 \text{ and } \lim\limits_{n\rightarrow \infty} n\Big( \tau_n((T_A+T_B)^4-4) \Big) =-4.
$$
Then, we summarize the first fourth limit laws and limit infinitesimal laws of $T_A+T_B$ as follows.
$$
m_k(T_A+T_B) = \begin{cases}
0 & \text{ if }k=1,3 \\
2 & \text{ if }k=2 \\
4 & \text{ if }k=4
\end{cases}
\text{ and } 
m'_k(T_A+T_B) = \begin{cases}
0 & \text{ if }k=1,3 \\
-2 & \text{ if }k=2 \\
-4 & \text{ if } k=4
\end{cases}.
$$
Note that the corresponding first four Boolean cumulants and infinitesimal Boolean cumulmants of $T_A$ are
\begin{eqnarray*}
b_1(T_A) &=& m_1(T_A) = 0 ; \\
b_2(T_A) &=& m_2(T_A)-m_1(T_A)^2 = 1 ;\\
b_3(T_A) &=& m_3(T_A)-2m_2(T_A)m_1(T_A)-m_1^3(T_A) = 0 ; \\
b_4(T_A) &=& m_4(T_A)-2m_3(T_A)m_1(T_A)-m^2_2(T_A)-3m_2(T_A)m_1^2(T_A)-m_1^4(T_A)=0,
\end{eqnarray*}
and 
\begin{eqnarray*}
b_1'(T_A) &=& m_1'(T_A) = 0 ; \\
b_2'(T_A) &=& m_2'(T_A)-2m_1(T_A)m_1'(T_A) = -1 ;\\
b_3'(T_A) &=& m_3'(T_A)-2(m_2'(T_A)m_1(T_A)+m_2(T_A)m_1'(T_A))-3m_1^2(T_A)m_1'(T_A) = 0 ; \\
b_4'(T_A) &=& m_4'(T_A)-2(m_3'(T_A)m_1(T_A)+m_3(T_A)m'_1(T_A))-2m_2(T_A)m_2'(T_A)\\
          & & -3(m_2'(T_A)m_1^2(T_A)+2m_2(T_A)m_1(T_A)m_1'(T_A))-4m_1^3(T_A)m_1'(T_A)=1.
\end{eqnarray*}
On the other hand, the first four Boolean cumulants and infinitesimal Boolean cumulmants of $T_A+T_B$ are
\begin{eqnarray*}
b_1(T_A+T_B) &=& m_1(T_A+T_B) = 0 ; \\
b_2(T_A+T_B) &=& m_2(T_A+T_B)-m_1(T_A+T_B)^2 = 2 ;\\
b_3(T_A+T_B) &=& m_3(T_A+T_B)-2m_2(T_A+T_B)m_1(T_A+T_B)-m_1^3(T_A+T_B)=0 ; \\
b_4(T_A+T_B) &=& m_4(T_A+T_B)-2m_3(T_A+T_B)m_1(T_A+T_B)-m^2_2(T_A+T_B) \\
             & & -3m_2(T_A+T_B)m_1^2(T_A+T_B)-m_1^4(T_A+T_B)=0,
\end{eqnarray*}
and also
\begin{eqnarray*}
b_1'(T_A+T_B) &=& m_1'(T_A+T_B) = 0 ; \\
b_2'(T_A+T_B) &=& m_2'(T_A+T_B)-2m_1(T_A+T_B)m_1'(T_A+T_B) = -2 ;\\
b_3'(T_A+T_B) &=& m_3'(T_A+T_B)-2(m_2'(T_A+T_B)m_1(T_A+T_B)+m_2(T_A+T_B)m_1'(T_A+T_B)) \\
              & & -3m_1^2(T_A+T_B)m_1'(T_A+T_B) = 0 ; \\
b_4'(T_A+T_B) &=& m_4'(T_A+T_B)-2(m_3'(T_A+T_B)m_1(T_A+T_B)+m_3(T_A+T_B)m'_1(T_A+T_B)) \\
              & & -3(m_2'(T_A+T_B)m_1^2(T_A+T_B)+2m_2(T_A+T_B)m_1(T_A+T_B)m_1'(T_A+T_B)) \\
              & & -2m_2(T_A+T_B)m_2'(T_A+T_B)-4m_1^3(T_A+T_B)m_1'(T_A+T_B)=4.
\end{eqnarray*}
From the above cumulants computation, we can see that for $k=1,\dots,4$, 
$$
b_k(T_A+T_B) = b_k(T_A)+b_k(T_B).
$$
However, 
$$b'_4(T_A+T_B)=4 \neq 2 = b'_4(T_A)+b'_4(T_B),$$ which implies that 
$T_A$ and $T_B$ are not asymptotically infinitesimally Boolean w.r.t partial trace $\tau_n$.

\bibliographystyle{abbrv}
\bibliography{main.bib}

\end{document}